\documentclass[reqno]{amsart}

\usepackage{enumerate,amsthm, amsmath,amssymb}
\usepackage{graphicx, hyperref, fullpage}
\usepackage{xcolor,ulem,psfrag}

\usepackage[latin1]{inputenc}
\definecolor{Ggreen}{RGB}{65,150,111} %

\renewcommand{\epsilon}{\varepsilon}
\renewcommand{\em}{\it}

\let\theta\vartheta
\let\phi\varphi
\let\<\langle
\let\>\rangle
\DeclareMathAlphabet{\doba}{U}{msb}{m}{n}
\gdef\mC{\doba{C}}

\gdef\mN{\doba{N}}
\gdef\mR{\doba{R}}
\gdef\mF{\doba{F}}
\gdef\mZ{\doba{Z}}

\def\supp{{\mathop{\rm supp}}}

\def\vo{{\mathop{\rm dvol}}}

\def\Id{\operatorname{Id}}
\def\a{{\mathfrak{a}}}

\newtheorem{lem}{Lemma}
\newtheorem{cor}[lem]{Corollary}
\newtheorem{thm}[lem]{Theorem}

\theoremstyle{definition}
\newtheorem{defi}[lem]{Definition}
\newtheorem{ex}[lem]{Example}
\newtheorem{rem}[lem]{Remark}

\newcommand{\uD}{\mathrm{D}}
\newcommand{\uA}{\mathcal{A}}
\newcommand{\uT}{\mathcal{T}}
\newcommand{\beq}{\begin{equation}}
\newcommand{\eeq}{\end{equation}}
\newcommand{\real}{\mathbb{R}}
\newcommand{\R}{\mathbb{R}}
\newcommand{\rn}{\mathbb{R}^n}
\newcommand{\Hsp}{H^s_p}

\newcommand{\nat}{\mathbb{N}}
\newcommand{\Tr}{\mathrm{Tr}}
\newcommand{\F}{F^s_{p,q}(\mathbb{R}^n)}

 \newcounter{mnotecount}[section]

\begin{document}

\title{Sobolev spaces on Riemannian manifolds with bounded geometry:\\ General coordinates and traces}
\author{Nadine Gro\ss{}e and Cornelia Schneider} 
\date{\today}

\maketitle

\begin{abstract}
We study fractional Sobolev and Besov spaces on noncompact Riemannian manifolds with bounded geometry. Usually, these spaces are  defined via geodesic normal coordinates which,  depending on the problem at hand, may often not be the best choice. We consider a more general definition  subject  to different local coordinates and give sufficient conditions on the corresponding coordinates resulting in  equivalent norms. Our main application is the computation of traces on submanifolds  with the help of Fermi coordinates. Our results also hold for corresponding spaces defined on vector bundles of bounded geometry and, moreover,  can be generalized to Triebel-Lizorkin spaces on manifolds,  improving \cite{Skr90}. 
\end{abstract}

\section{Introduction}

The main aim of  this paper is to consider fractional Sobolev spaces on noncompact Riemannian manifolds, equivalent characterizations of these spaces and their traces on submanifolds.  We  address the problem to what extend  results from classical analysis on Euclidean space carry over to the setting of Riemannian manifolds -- without making any unnecessary assumptions about the manifold. 
In particular, we will be interested in  noncompact manifolds since the compact case presents no difficulties and is well understood.  \\

Let $(M,g)$ denote an $n$-dimensional, complete, and noncompact Riemannian manifold with Riemannian metric $g$. 
Fractional Sobolev spaces on manifolds $H^s_p(M)$,  $s\in \real$, $1<p<\infty$,  can be defined similar to corresponding Euclidean spaces $H^s_p(\rn)$, usually characterized via
\[
H^s_p=(\mathrm{Id} -\Delta)^{-s/2}L_p,
\]
by replacing the Euclidean Laplacian $\Delta$ with the Laplace-Beltrami operator on $(M,g)$  and using an auxiliary parameter $\rho$, see Section \ref{3.1}. The spaces $H^s_p(M)$ were introduced and studied in detail in \cite{strich} 
and  generalize in a natural way classical Sobolev spaces on manifolds, $W^k_p(M)$, which contain all $L_p$ functions on $M$ having bounded covariant derivatives up to order $k\in \nat$,  cf. \cite{Aub1,Aub2}. \\
{To avoid any confusion, let us emphasize that in this article we  study exactly these  fractional Sobolev spaces $H^s_p(M)$ defined by means of powers of $\Delta$.  But we shall use   an alternative characterization  of these spaces on manifolds with bounded geometry as definition -- having in mind the proof of our main theorem.\\
To be more precise,}  on  manifolds with bounded geometry, see Definition \ref{bdd_geo},  one can alternatively define fractional Sobolev spaces $H^s_p(M)$ via localization and pull-back onto $\rn$, by using geodesic normal coordinates and  corresponding fractional Sobolev spaces on $\R^n$, cf. \cite[Sections~7.2.2,~7.4.5]{Tri92} and also \cite[Definition~1]{Skr98}. Unfortunately, for some applications the choice of geodesic normal coordinates is not convenient, which is why we do not wish to restrict ourselves to these coordinates only.  The main application we have in mind are traces on submanifolds $N$ of $M$. But also  for manifolds with symmetries, product manifolds or warped products, geodesic normal coordinates may not be the first and natural choice and one is interested in coordinates better suited to the problem at hand. 

Therefore,  we  introduce in Definition \ref{H-koord}  Sobolev spaces $H^{s,\uT}_p(M)$ in a more general way,  containing all those complex-valued distributions $f$ on $M$ such that 
\beq\label{intro-1}
\Vert f\Vert_{H^{s,\uT}_{p}}:=\left(\sum_{\alpha\in I} \Vert (h_\alpha f)\circ \kappa_\alpha\Vert^p_{H^s_p(\mathbb{R}^n)}\right)^{1/p}
\eeq
is finite, where $\uT=(U_{\alpha},\kappa_{\alpha},h_{\alpha})_{\alpha\in I}$  denotes a trivialization of $M$ consisting of a uniformly locally finite covering $U_{\alpha}$, local coordinates $\kappa_{\alpha}:V_{\alpha}\subset \rn\rightarrow U_{\alpha}\subset M$ (not necessarily geodesic normal coordinates) and a subordinate partition of unity $h_{\alpha}$.  Of course,  the case of local coordinates $\kappa_{\alpha}$  being geodesic normal coordinates is covered but we can choose from a larger set of trivializations.  Clearly, we are not interested in all $\uT$ but merely the so called admissible trivializations $\uT$, cf.  Definition \ref{bddcoord}, yielding the coincidence  $$H^{s,\uT}_p(M)=H^s_p(M),$$ cf. Theorem \ref{indep_H}.

As pointed out earlier, our main applications in mind are Trace Theorems. In \cite[Theorem 1]{Skr90}, traces on manifolds were studied using the Sobolev norm \eqref{intro-1} with geodesic normal coordinates. Since these coordinates in general do not take into account the structure of the underlying submanifold where  the trace is taken, one is limited to so-called {\em geodesic} submanifolds. This is highly restrictive, since geodesic submanifolds are very exceptional. Choosing coordinates that are more adapted to the situation will immediately enable us to compute the trace on a much larger class of submanifolds. In particular, we consider Riemannian manifolds $(M,g)$ with submanifolds $N$ such that $(M,N)$ is of bounded geometry, see Definition \ref{bdd_geo}, i.e., $(M,g)$ is of bounded geometry, the mean curvature of $N$ and its covariant derivatives are uniformly bounded, the injectivity radius of $(N,g_N)$ is positive and  there is a uniform collar of $N$.

The coordinates of choice for proving Trace Theorems  are Fermi coordinates,  introduced in Definition \ref{FC}. We show in Theorem \ref{FC_admis} that for a certain cover with Fermi coordinates there is a subordinated partition of unity such that the resulting trivialization is admissible.

The main Trace Theorem itself is stated in  Theorem \ref{trace-th}, where we prove that if 
 $M$ is a manifold of  dimension $n\geq 2$, $N$  a submanifold of dimension $k<n$, and   $(M,N)$  of bounded geometry, we have for  $s>\frac{n-k}{p}$,
\beq\label{intro-2}\Tr_N \; H^s_p(M)= B^{s-\frac{n-k}{p}}_{p,p}(N).
\eeq
i.e., there is a linear, bounded and surjective trace operator $\Tr_N$ with a linear and bounded right inverse $\text{Ex}_M$ from the trace space into the original space such that $\Tr_N\circ \text{Ex}_M=\Id$, where $\Id$ denotes the identity on operator $N$.
The spaces on the right hand side of \eqref{intro-2} are  Besov spaces obtained via real interpolation of the spaces $H^s_p$, cf. Remark \ref{rem-B}. When just asking for $\Tr_N$ to be linear and bounded, one can reduce the assumptions on $(M,N)$ further by replacing the existence of a collar of $N$ with a uniform local collar, cf. Remark \ref{rem_th_loccoll}.
 
We believe that the method presented in this article is very well suited to tackle the trace problem on manifolds. One could also think of computing traces using atomic decompositions of the spaces $H^s_p(M)$ as established in  \cite{Skr98}, which is often done when dealing with traces on hyperplanes of $\rn$ or on domains. But on (sub-)manifolds  it should be complicated (if not impossible) to   obtain a linear and continuous extension operator from the trace space into the source space --  which by our method follows immediately from corresponding results on $\rn$. \\

In Section \ref{sec_vec_bu}, we establish analogous results for vector bundles of bounded geometry. An application of our trace result for vector bundles, Theorem \ref{trace-th_vec}, may be found in \cite{GN}, where the authors classify boundary value problems of the Dirac operator on $\text{spin}^{\mC}$ bundles of  bounded geometry, deal with the existence of a solution, and obtain some spectral estimates for the Dirac operator on hypersurfaces of bounded geometry.\\

As another application of our general coordinates spaces with symmetries are considered in  Section \ref{outlook_1}.  We restrict ourselves to the straight forward  case where the symmetry group is discrete and obtain a generalization of a theorem from  \cite[Section 9.2.1]{T-F1}, where the author characterizes Sobolev spaces on the  tori $\mathbb{T}^n:=\rn/ \mathbb{Z}^n$  via  weighted Sobolev spaces on $\mR^n$ containing $\mZ^n$ periodic distributions only.\\

Finally, in Section \ref{sec_TL}  we deal with the larger scale of Triebel-Lizorkin spaces $F^{s,\uT}_{p,q}(M)$, $s\in \real$, $0<p<\infty$, $0<q\leq \infty$ or $p=q=\infty$, linked with  fractional Sobolev spaces via 
\[
F^{s,\uT}_{p,2}(M)=H^{s,\uT}_p(M), \qquad s\in \real, \quad 1<p<\infty, 
\]
and the general scale of Besov spaces $B^{s,\uT}_{p,q}(M)$, $s\in \real$, $0<p,q\leq \infty$ defined via real interpolation of the spaces $F^{s,\uT}_{p,q}(M)$, cf. Definition \ref{F-koord}. We will show that an admissible trivialization $\uT$ again guarantees coincidence with the corresponding spaces $F^s_{p,q}(M)$, $B^s_{p,q}(M)$ --  obtained from choosing geodesic normal coordinates, cf. \cite[Sections 7.2, 7.3]{Tri92} -- and that  trace results from Euclidean space carry over to our setting of submanifolds $N$ of $M$, where $(M,N)$ is of bounded geometry. In particular,  if now
\begin{equation}\label{intro-3}
s-\frac{n-k}{p}>k \max\left(0,\frac 1p-1\right), 
\end{equation} 
we have 
\begin{equation*}
\Tr \ F^s_{p,q}(M)=B^{s-\frac{n-k}{p}}_{p,p}(N) \qquad \text{\ and\ }\qquad 
\Tr \ B^s_{p,q}(M)=B^{s-\frac{n-k}{p}}_{p,q}(N), 
\end{equation*}
cf. Theorem \ref{trace-th-gen}. The restriction \eqref{intro-3} is natural and best possible also in the  Euclidean case. \\

{\bf Acknowledgement.} { We are grateful to Sergei V. Ivanov who kindly answered our question on mathoverflow concerning the equivalence of different characterizations on manifolds of bounded geometry.} Moreover, we thank Hans Triebel for  helpful discussions on the subject. The second author thanks the University of Leipzig  for the hospitality and support during a short term visit in Leipzig.

\section{Preliminaries and notations}

\paragraph{\bf General notations.}  Let $\nat$ be the collection of all natural numbers, and let $\nat_0 = \nat \cup \{0 \}$. Let $\rn$ 
be the $n$-dimensional Euclidean space, $n \in \nat$, $\mathbb{C}$ the complex plane, and let $B_r^n$  denote the ball in $\mR^n$ with center $0$ and radius $r$ {(sometimes simply denoted by $B_r$ if there is no danger of confusion)}. Moreover, index sets are always assumed to be countable, and we use the Einstein sum convention.

Let the standard coordinates on $\mR^n$ be denoted by $x=(x^1,x^2,\ldots, x^n)$.
The partial derivative operators in direction of the coordinates are denoted by $\partial_i=\partial/\partial x^i$ for $1\leq i\leq n$.  
The set of multi-indices $\a=(\a_1, \dots, \a_n)$, $\a_i\in\nat_0$, $i=1, \dots,
n$, is denoted by $\nat_0^n$, and we shall use the common notation
$
\uD^\a f=\partial_1^{\a_1}...\partial_n^{\a_n}f =\frac{\partial^{|\a|}f}{(\partial x^1)^{\a_1} \cdots (\partial x^n)^{\a_n}},$ where $f$ is a function on $\mR^n$. 
As usual, let $|\a|=\a_1 + \cdots + \a_n$ be the order of the derivative $\uD^\a f$. 
Moreover, we put $x^\a=(x^1)^{\a_1} \cdots (x^n)^{\a_n}$.\\

For a real number $a$, let $a_+:=\max(a,0)$, and let $[a]$ denote its
integer part. For $p\in (0,\infty]$, the number $p'$ is defined by
$1/p':=(1-1/p)_+$ with the convention that $1/\infty=0$. All unimportant positive constants will be denoted by $c$, occasionally with
subscripts.  For two non-negative expressions ({{\it i.e.},
functions or functionals) ${\mathcal  A}$, ${\mathcal  B}$, the
symbol ${\mathcal A}\lesssim {\mathcal  B}$ (or ${\mathcal A}\gtrsim
{\mathcal  B}$) means that $ {\mathcal A}\leq c\, {\mathcal  B}$ (or
$c\,{\mathcal A}\geq {\mathcal B}$) for a suitable constant $c$. If ${\mathcal  A}\lesssim
{\mathcal  B}$ and ${\mathcal A}\gtrsim{\mathcal  B}$, we write
${\mathcal  A}\sim {\mathcal B}$ and say that ${\mathcal  A}$ and
${\mathcal  B}$ are equivalent. Given two (quasi-) Banach spaces $X$ and $Y$, we write $X\hookrightarrow Y$
if $X\subset Y$ and the natural embedding of $X$ into $Y$ is continuous.\\ 

\paragraph{\bf Function spaces on $\rn$.} $L_p(\rn)$, with $0<p\leq\infty$, stands for the usual quasi-Banach space with respect to the Lebesgue measure, quasi-normed by
\[
\Vert f \Vert_{L_p(\rn)}:=\left(\int_{\rn}|f(x)|^p\mathrm{d} x\right)^{\frac 1p}
\]
with the usual modification if $p=\infty$. For $p\geq 1$, $L_p(\mR^n)$ is even a Banach space. Let $\mathcal{D}(\rn)$} denote the space of smooth functions with compact support, and let $\mathcal{D}'(\mR^n)$ denote the corresponding distribution space.  
By $\mathcal{S}(\rn)$ we denote the Schwartz
space of all complex-valued rapidly
decreasing infinitely differentiable functions on $\rn$ and by $\mathcal{S}'(\rn)$ the dual space of all tempered distributions on $\rn$. For a rigorous definition of the Schwartz space and 'rapidly decreasing' we refer to \cite[Section 1.2.1]{T-F1}. 
For $f\in \mathcal{S}'(\rn)$ we denote by
$\widehat{f}$  the Fourier transform of $f$ and by  $f^{\vee}$   the
inverse Fourier transform of $f$. \\
Let  $s\in \real$ and $1<p<\infty$. Then
the  $(${\em fractional}$)$ {\em Sobolev space}  $H^s_p(\rn)$  contains all $f\in \mathcal{S}'(\rn)$ with \[\big((1+|\xi|^2)^{s/2}\widehat{f}\; \big)^\vee\in L_p(\rn),\qquad \xi\in \rn,\]
cf. \cite[Section 1.3.2]{Tri92}.
In particular, for $k\in \mathbb{N}_0$, these spaces coincide with  the {\em classical Sobolev spaces} $W^k_p(\rn)$, 
\[H^k_{p}(\rn)=W^k_{p}(\rn),\qquad \text{i.e.,}\qquad H^0_{p}(\rn)=L_p(\rn),\]
usually normed by
$$\Vert f\Vert_{W^k_p(\rn)}=\left(\sum_{|\a|\leq k}\Vert\mathrm{D}^{\a}f\Vert_{L_p(\rn)}^p\right)^{1/p}.$$

Furthermore, Besov spaces $B^s_{p,p}(\rn)$ can be defined via interpolation of Sobolev spaces. 
In particular, let  $(\cdot , \cdot )_{\Theta, p}$ stand for the real interpolation method, cf. \cite[Section 1.6.2]{Tri92}.  Then  for $s_0,s_1\in \real$,  $1<p<\infty$, and $0<\Theta<1$, we put 
\ \mbox{$
B^{s}_{p,p}(\rn):=\left(H^{s_0}_p(\rn), H^{s_1}_p(\rn)\right)_{\Theta,p},
$}\ 
where $s=\Theta s_0+(1-\Theta)s_1$. Note that $B_{p,p}^s(\rn)$ does not depend on the choice of $s_0,s_1, \Theta$.

The following lemma about pointwise multipliers and diffeomorphisms may be found in  \cite[Sections~4.2,4.3]{Tri92}, where it was proven in a more general setting.

\begin{lem}\label{Sob_Rn} Let  $s\in \mathbb{R}$ and  $1<p<\infty$. 

\begin{itemize}
 \item[(i)] Let $f\in H_p^s(\R^n)$ and $\phi$ a smooth function on $\R^n$ such that for all $\a$ with $|\a|\leq {[s]}+1$  we have $|\uD^\a \phi|\leq C_{|\a|}$. 
 Then there is a constant $C$ only depending on $s,p,n$ and $C_{|\a|}$ such that
\[ \Vert \phi f\Vert_{H_p^s(\R^n)} \leq C \Vert f\Vert_{H_p^s(\R^n)}.\]
\item[(ii)] Let $f\in H_p^s(\mR^n)$ with $\supp\, f\subset U\subset \R^n$ for $U$ open and let  $\kappa: V\subset \R^n\to U\subset \R^n$ be a diffeomorphism such that for all $\a$ with $|\a|\leq [s]+1$  we have $|\uD^\a \kappa|\leq C_{|\a|}$.Then there is a constant $C$ only depending on $s,p,n$ and $C_{|\a|}$ such that 
\[ \Vert f\circ \kappa \Vert_{H_p^s(\R^n)} \leq C \Vert f\Vert_{H_p^s(\R^n)}.\]
 \end{itemize}
\end{lem}

\paragraph{\bf Vector-valued function spaces on $\mR^n$.}

Let $\mathcal{D}(\mR^n, \mF^r)$ be the space of compactly supported smooth functions on $\mR^n$ with values in $\mF^r$ where $\mF$ stands for $\mR$ or $\mC$ and $r\in \mN$ . Let $\mathcal{D}'(\mR^n, \mF^r)$ denote the corresponding distribution space. Then, $H^s_p(\mR^n, \mF^r)$ is defined in correspondence with $H^s_p(\mR^n)$  from above, cf. [Triebel, Fractals and spectra, Section~15]. Moreover, Besov spaces $B^s_{p,p}(\rn, \mF^r)$ are defined as the spaces $B_{p,p}^s(\mR^n)$ from above;
$B^{s}_{p,p}(\rn, \mF^r):=\left(H^{s_0}_p(\rn, \mF^r), H^{s_1}_p(\rn, \mF^r)\right)_{\Theta,p}$ where $(\cdot, \cdot)_{\Theta,p}$ again denotes the real interpolation method with  $s_0,s_1\in \real$, $1<p<\infty$, and $0<\Theta<1$ with $s=\Theta s_0+(1-\Theta)s_1$.

\begin{lem}\label{vec_norm_equ}
 The norms
$\Vert \phi \Vert_{H_p^s(\mR^n, \mF^r)}$  and $\left( \sum_{i=1}^r \Vert  \phi_i \Vert_{H_p^s(\mR^n, \mF)}^p \right)^\frac{1}{p}$
are equivalent where $\phi=(\phi_1, \ldots, \phi_r)\in H_p^s(\mR^n, \mF^r)$. The analogous statement is true for Besov spaces.
\end{lem}

\begin{proof} The equivalence for Sobolev spaces follows immediately from their definition. The corresponding result for  Besov spaces can be found in \cite[Lemma 26]{Gro}.
\end{proof}

\paragraph{\bf Notations concerning manifolds.} Before starting  we want to make the following warning or excuse: For a differential geometer the notations may seem a little overloaded at first  glance. Usually, when interested in equivalent norms, one merely suppresses diffeomorphisms as transition functions. This provides no problem when it is clear that all constants appearing  are uniformly bounded --  which is obvious for finitely many bounded charts (on closed manifolds) and also known for manifolds of bounded geometry with geodesic normal coordinates. But here we work in a more general context where the aim is to find out which conditions the coordinates have to satisfy in order to  ignore those diffeomorphisms in the sequel. This is precisely  why we try to be more explicit in our notation.\\

Let $(M^n, g)$ be an $n$-dimensional complete manifold with Riemannian metric $g$. 
We denote the volume element on $M$ with respect to the metric $g$ by $\vo_g$. For $1<p<\infty$ the $L_p$-norm of a compactly supported smooth function $v\in \mathcal{D}(M)$ is given by $\Vert v\Vert_{L_p(M)}=\left( \int_M |v|^p\vo_g\right)^\frac{1}{p}$. The set $L_p(M)$ is then the completion of $\mathcal{D}(M)$ with respect to the $L_p$-norm. The space of distributions on $M$ is denoted by $\mathcal{D}'(M)$.

A cover $(U_\alpha)_{\alpha \in I}$ of $M$ is a collection of open subsets of $U_\alpha\subset M$ where $\alpha$ runs over an index set $I$.  The cover is called locally finite if each $U_\alpha$ is intersected by at most finitely many $U_\beta$.  The cover is called uniformly locally finite if there exists a constant $L>0$ such that each $U_\alpha$ is intersected by at most $L$ sets $U_\beta$. \\
A chart on $U_{\alpha}$ is given by local coordinates -- a diffeomorphism  $\kappa_\alpha: x=(x^1, \ldots, x^n) \in V_\alpha \subset \mathbb{R}^n \to \kappa_\alpha(x)\in U_\alpha$. We will always assume our charts to be smooth. A collection $\uA=(U_\alpha, \kappa_\alpha)_{\alpha\in I}$ is called an atlas of $M$. 

Moreover, a collection of smooth functions $(h_\alpha)_{\alpha\in I}$ on $M$ with 
$$ \supp\ h_{\alpha}\subset U_{\alpha},\qquad  0\leq h_\alpha \leq 1\qquad  \text{and} \qquad  \sum_\alpha h_\alpha=1 \quad \text{on }M. 
 $$ 
 is called a partition of unity subordinated to the cover $(U_\alpha)_{\alpha\in I}$.
 The triple $\uT:=(U_{\alpha}, \kappa_{\alpha},h_{\alpha})_{\alpha \in I}$ is called a trivialization of the manifold $M$.\\

Using the standard Euclidean coordinates $x=(x^1, \ldots, x^n)$ on $V_\alpha\subset \mR^n$, we introduce an orthonormal frame $(e_i^{\alpha})_{1\leq i\leq n}$ on $TU_\alpha$ by $e_i^\alpha:=(\kappa_\alpha)_*(\partial_{i})$. In case we talk about a fixed chart we will often leave out the superscript $\alpha$. Then, in those local coordinates the metric $g$ is expressed via the matrix coefficients $g_{ij}(=g_{ij}^\alpha): V_\alpha \to \mathbb{R}$ defined by $g_{ij}\circ \kappa_\alpha^{-1}=g(e_i,e_j)$  and the corresponding Christoffel symbols $\Gamma_{ij}^k=\!\! ({\phantom{!}}^{\alpha\mkern-1mu}\Gamma_{ij}^k): V_\alpha \to \mR$ are  defined by $\nabla^M_{e_i}e_j =(\Gamma_{ij}^k\circ \kappa_\alpha^{-1}) e_k$ where $\nabla^M$ denotes the Levi-Civita connection of $(M,g)$. In local coordinates, 
\begin{equation} \Gamma_{ij}^k=\frac{1}{2} g^{kl}(\partial_j g_{il}+\partial_i g_{jl} -\partial_l g_{ij}) \label{Christ_coord}
\end{equation}
where $g^{ij}$ is the inverse matrix of $g_{ij}$. If $\alpha,\beta\in I$ with $U_\alpha\cap U_\beta\neq \varnothing$, we define the transition function $\mu_{\alpha\beta}=\kappa_\beta^{-1}\circ \kappa_\alpha: \kappa_\alpha^{-1}(U_\alpha\cap U_\beta) \to \kappa_\beta^{-1}(U_\alpha\cap U_\beta)$. Then, 
\begin{equation}\label{trans_g} g_{ij}^{\alpha}(x)= \partial_i \mu_{\alpha\beta}^k(x)\partial_j \mu_{\alpha\beta}^l(x) g_{kl}^\beta(\mu_{\alpha\beta}(x)).\end{equation}

\begin{ex}[{\bf Geodesic normal coordinates}]\label{geod_coord_triv}
Let $(M^n,g)$ be a complete Riemannian manifold. Fix $z\in M$ and let $r>0$ be smaller than the injectivity radius of $M$. For $v\in T_z^{\leq r}M:=\{ w\in T_zM\ |\ g_z(w,w)\leq r^2\}$, we denote by $c_v: [-1,1] \to M$ the unique geodesic with $c_v(0)=z$ and $\dot{c}_v(0)=v$. Then, the exponential map $\exp^M_z: T_z^{\leq r}M \to M$ is a diffeomorphism defined by
$\exp^M_z(v):= c_{v}(1)$. Let $S=\{p_\alpha\}_{\alpha\in I}$ be a set of points in $M$ such that $(U^{\rm geo}_\alpha:= B_r(p_\alpha))_{\alpha\in I}$ covers $M$.  For each $p_\alpha$ we choose an orthonormal frame of $T_{p_\alpha}M$ and call the resulting identification $\lambda_\alpha: \mR^n\to T_{p_\alpha} M$. Then, $\uA^{\rm geo}= (U^{\rm geo}_\alpha, \kappa^{\rm geo}_\alpha=\exp_{p_\alpha}^M\circ \lambda_\alpha: V_\alpha^{\rm geo}:=B_r^n \to U_\alpha^{\rm geo})_{\alpha \in I}$ is an atlas of $M$ -- called geodesic atlas. { (Note that $\lambda_\alpha^{-1}$ equals the tangent map ${(d \kappa_\alpha^{\rm geo})}^{-1}$ at $p_\alpha$.)}
\end{ex}

\paragraph{\bf Notations concerning vector bundles.}

Let $E$ be a hermitian or Riemannian vector bundle over a Riemannian manifold $(M^n,g)$ of rank $r$ with fiber product $\<.,.\>_E$ and  connection $\nabla^E: \Gamma(TM)\otimes \Gamma(E)\to \Gamma(E)$. Here $\Gamma$ always denotes the space of smooth sections of the corresponding vector bundle.
We set $\mF=\mR$ if $E$ is a Riemannian vector bundle and $\mF=\mC$ if $E$ is hermitian.

Let $\uA=(U_\alpha, \kappa_\alpha: V_\alpha \to U_\alpha)_{\alpha \in I}$ be an atlas of $(M,g)$ and 
let $\zeta_\alpha: U_\alpha \times \mF^r \to E|_{U_\alpha}$ be local trivializations of $E$. Note that here 'trivialization' has the usual meaning in connection with the ordinary definition of a vector bundle. We apologize that in lack of a better notion we also call $\uT$  a trivialization but hope there will be no danger of confusion.  We set $\xi_\alpha:= \zeta_\alpha \circ (\kappa_\alpha \times \Id) : V_\alpha \times \mF^r \to  E|_{U_\alpha}$.
We call $\uA_E=(U_\alpha, \kappa_\alpha, \xi_\alpha)_{\alpha \in I}$ an atlas of $E$. In case we already start with a trivialization $\uT=(U_\alpha, \kappa_\alpha, h_\alpha)_{\alpha \in I}$ on $M$, $\uT_E=(U_\alpha, \kappa_\alpha, \xi_\alpha, h_\alpha)_{\alpha \in I}$ is called a trivialization of $E$. 

Let $y=(y^1,\ldots, y^r)$ be standard coordinates on $\mF^r$ and let $\left(\partial_\rho:=\frac{\partial}{\partial y^\rho}\right)_{1\leq \rho\leq r}$ be the corresponding local frame. Then, $\tilde{e}_\rho(p)(=\tilde{e}_\rho^\alpha(p)):=\xi_\alpha \left(\kappa_\alpha^{-1}(p),\partial_\rho \right)$ form a local frame of $E_{p}$ for $p\in U_\alpha$. As before, we suppress $\alpha$ in the notation if we talk about a fixed chart. In those local coordinates, the fiber product is represented by $h_{\rho\sigma}:=\langle \tilde{e}_\rho, \tilde{e}_\sigma\rangle_E\circ \kappa_\alpha: V_\alpha \to \mF$. Hence, if  $\phi,\psi\in \Gamma(E|_{U_\alpha})$ we have for $\phi=\phi^\rho \tilde{e}_\rho$ and $\psi=\psi^\sigma \tilde{e}_\sigma$ that
\[ \langle \phi,\psi\rangle_E= (h_{\rho\sigma} \circ \kappa_\alpha^{-1}) \phi^\rho \bar{\psi}^\sigma,\]
where $\bar{a}$ denotes the complex conjugate of $a$.
Let  Christoffel symbols $\tilde{\Gamma}_{i\rho}^\sigma: U_\alpha \to \mF$ for $E$ be defined by $\nabla^E_{e_i} \tilde{e}_\rho =\left(\tilde{\Gamma}_{i\rho}^\sigma \circ \kappa_\alpha^{-1}\right) \tilde{e}_\sigma$, where $e_i=(\kappa_\alpha)_*\partial_i$.
If the connection $\nabla^E$ is metric, i.e., $e_i\< \tilde{e}_\sigma,\tilde{e}_\rho\>_E=\< \nabla^E_{e_i}\tilde{e}_\sigma,\tilde{e}_\rho\>_E+\< \tilde{e}_\sigma,\nabla_{e_i} \tilde{e}_\rho\>_E$, we get
\begin{equation}\label{h_ode} 
\partial_i h_{\sigma\tau} =  \Gamma_{i\sigma}^\rho h_{\tau\rho} + \Gamma_{i\tau}^\rho h_{\rho\sigma}. \end{equation}

For all $\alpha, \beta\in I$ with $U_\alpha \cap U_\beta\neq \varnothing$, transition functions $\tilde{\mu}_{\alpha\beta}: \kappa_\alpha^{-1}(U_\alpha\cap U_\beta) \to \text{GL}(r,\mF)$ are defined by $\xi_\beta^{-1}\circ \xi_\alpha (x,u)= (\mu_{\alpha\beta}(x), \tilde{\mu}_{\alpha\beta} (x) \cdot u)$. 
Here, $\text{GL}(r,\mF)$ denotes the general linear group of $\mF$-valued $r\times r$ matrices. \\

\paragraph{\bf Flows.}

Let $x'(t)=F(t,x(t))$ be a system of ordinary differential equations with $t\in\mR$, $x(t)\in \mR^n$ and $F\in C^\infty (\mR\times \mR^n, \mR^n)$. Let the solution of the initial value  problem $x'(t)=F(t,x(t))$ with $x(0)=x_0\in \mR^n$ be denoted by $x_{x_0}(t)$ and exist for $0\leq t\leq t_0(x_0)$. Then, the flow $\Phi: {\rm dom}\subset  \mR\times \mR^n\to \mR^n$ with ${\rm dom}\subset \{ (t,x)\ |\ 0\leq t\leq t_0(x)\}$ is defined by  $\Phi (t,x_0)=x_{x_0}(t)$.
Higher order ODE's $x^{(d)}(t)=F(t,x(t),\ldots, x^{(d-1)}(t))$ can be transferred back to first order systems by introducing auxiliary variables. The corresponding flow then obviously depends not only on $x_0=x(0)$ but the initial values $x(0), x'(0),\ldots, x^{(d-1)}(0)$: $\Phi(t, x(0), \ldots, x^{(d-1)}(0))$.

\begin{ex}[{\bf Geodesic flow}]\label{geo_flow} Let $(M^n,g)$ be a Riemannian manifold. Let $z\in M$, $v\in T_zM$. Let $\kappa: V\subset \mR^n\to U\subset M$ be a chart around $z$.  The corresponding coordinates on $V$ are denoted by  $x=(x^1,\ldots, x^n)$. We consider the geodesic equation in coordinates: $\ddot{x}^k=-\Gamma_{ij}^k {\dot{x}^i}{\dot{x}^j}$ with initial values $x(0)=\kappa^{-1}(z)\in \R^n$ and {$x'(0)=\kappa^*(v)(=d \kappa^{-1}(v))$.} Here $\Gamma_{ij}^k$ are the Christoffel symbols with respect to the coordinates given by $\kappa$. Let $x(t)$ be the unique solution and $\Phi(t, x(0), x'(0))$ denotes the corresponding flow. Then, $c_v(t)=\kappa(x(t))$ is the geodesic described in Example \ref{geod_coord_triv} and { $\exp^M_z(v)=\kappa\circ \Phi(1,\kappa^{-1}(z), \kappa^*(v))$. }
\end{ex}

\begin{lem}\cite[Lemma 3.4 and Corollary 3.5]{Schick01}\label{flow_lem}
 Let $x'(t)=F(t,x(t))$ be a system of ordinary differential equations as above. Suppose that $\Phi(t,x)$ is the flow of this equation. Then there is a universal expression ${\rm Expr}_\a$ only depending on the multi-index $\a$ such that  
 \[ |\uD^\a_x \Phi(t,x_0)|\leq {\rm Expr}_\a \left(\sup_{0\leq \tau\leq t} \left\{\left|D^{\a'}_x F(\tau, \Phi(\tau, x_0))\right|\right\} \ \Big|\ \a'\leq \a,\ t\right)\] for all $t\geq 0$ where $\Phi(t, x_0)$ is defined.
 Moreover, a corresponding statement holds for ordinary differential equations of order $d$.
\end{lem}

\section{Sobolev spaces on manifolds of bounded geometry}

From now on let $M$ always be an $n$-dimensional  manifold with Riemannian metric $g$.

\begin{defi}\label{defi_bdgeom}\cite[Definition A.1.1]{Shu}
A Riemannian manifold $(M^n,g)$ is of bounded geometry if the following two conditions are satisfied:
\begin{itemize}
\item[(i)] The injectivity radius $r_{M}$  of $(M,g)$ is positive.
\item[(ii)] Every covariant derivative of the Riemann curvature tensor $R^M$ of $M$ is bounded, i.e., for all $k\in \mN_0$ there is a constant $C_k>0$ such that $|(\nabla^M)^kR^M|_g\leq C_k$.
\end{itemize}
\end{defi}

\begin{rem}\label{rem_bdgeom}{\bf i)} Note that Definition \ref{defi_bdgeom}(i) implies that $M$ is complete, cf. \cite[Proposition 1.2a]{Eichb}.\\
{\bf ii)}\cite[Definition A.1.1 and below]{Shu} Property (ii) of Definition \ref{defi_bdgeom} can be replaced by the following equivalent property which will be more convenient later on: Consider a geodesic atlas $\uA^{\rm geo}=(U^{\rm geo}_{\alpha}, \kappa^{\rm geo}_{\alpha})_{\alpha \in I}$ as in Example \ref{geod_coord_triv}. For  all $k\in \mN$ there are constants $C_k$ such that for all  $\alpha,\beta\in I$ with $U^{\rm geo}_{\alpha}\cap U^{\rm geo}_{\beta}\neq \varnothing$ we have for the corresponding transition functions $\mu_{\alpha\beta}:=(\kappa^{\rm geo}_{\beta})^{-1}\circ \kappa^{\rm geo}_{\alpha}$ that 

\begin{equation*}
|\mathrm{D}^{\a} \mu_{\alpha\beta}| \leq C_{k},\qquad \text{for all $\a \in \nat^n_0$ with $|\a|\leq k$ and all charts.}
\end{equation*}

{\bf iii)} \cite[Theorem A and below]{Eich}  Consider a geodesic atlas $\uA^{\rm geo}$ as above. Let $g_{ij}$ denote the metric in these coordinates and $g^{ij}$ its inverse. Then, property (ii) of Definition \ref{defi_bdgeom} can be replaced by the following equivalent property: For all $k\in \mN_0$ there is a constant $C_k$ such that

\begin{equation}\label{comp-coord}
|\mathrm{D}^{\a} g_{ij}| \leq C_{k},\ |\mathrm{D}^{\a} g^{ij}| \leq C_{k}, \qquad \text{for all $\a\in \mN_0^n$ with $|\a|\leq k$. }
\end{equation}
\end{rem}

\begin{ex}[\bf Geodesic trivialization]\label{geod_triv}
 Let $(M,g)$ be {of bounded geometry (this includes the case of closed manifolds)}. Then, there exists a geodesic atlas, see Example \ref{geod_coord_triv}, that is uniformly locally finite: Let $S$ be a maximal set of points $\{p_\alpha\}_{\alpha \in I}\subset M$ such that the metric balls $B_{\frac{r}{2}}(p_\alpha)$ are pairwise disjoint. Then, the balls $\{B_r(p_\alpha)\}_{\alpha \in I}$ cover $M$, and  we obtain a (uniformly locally finite) geodesic atlas $\uA^{\rm geo}=(U^{\rm geo}_\alpha:= B_r(p_\alpha), \kappa^{\rm geo}_\alpha)_{\alpha\in I})$. {For an argument concerning the uniform local finiteness of the cover we refer to Remark \ref{rem_FC}.ii.}
Moreover, there is a partition of unity $h^{\rm geo}_\alpha$ subordinated to $(U^{\rm geo}_\alpha)_{\alpha\in I}$ such that for all $k\in \mN_0$ there is a constant $C_k>0$ such that $|\uD^\a (h^{\rm geo}_\alpha\circ \kappa^{\rm geo}_\alpha)|\leq C_k$ for all multi-indices $\a$ with $|\a|\leq k$, cf. \cite[Proposition~7.2.1]{Tri92} and the references therein.
The resulting trivialization is denoted by $\uT^{\rm geo}=(U^{\rm geo}_\alpha, \kappa^{\rm geo}_\alpha, h^{\rm geo}_\alpha)_{\alpha \in I}$ and referred to as geodesic trivialization.
\end{ex}

\subsection{Sobolev norm on manifolds of bounded geometry using geodesic normal coordinates}\label{3.1}

On manifolds of bounded geometry it is  possible to define  spaces $H^s_p(M)$  using local descriptions (geodesic normal coordinates) and norms of corresponding spaces $H^s_p(\rn)$.\\

\begin{defi}\label{H-geo}
Let $(M^n,g)$  be a Riemannian manifold of bounded geometry with geodesic trivialization $\uT^{\rm geo}=(U^{\rm geo}_\alpha, \kappa^{\rm geo}_\alpha, h^{\rm geo}_\alpha)_{\alpha\in I}$  as above. Furthermore, let $s\in \mR$ and $1<p<\infty$. Then the space $H^{s}_p(M)$ contains all distributions $f\in \mathcal{D}'(M)$ such that 
\begin{equation}\label{geo_norm} \left( \sum_{\alpha\in I} \Vert (h^{\rm geo}_\alpha f)\circ \kappa^{\rm geo}_\alpha  \Vert_{H_p^s(\mR^n)}^p \right)^\frac{1}{p}\end{equation} is finite. Note that although $\kappa^{\rm geo}_\alpha$ is only defined on $V^{\rm geo}_\alpha\subset \mR^n$,  $(h^{\rm geo}_\alpha f)\circ \kappa^{\rm geo}_\alpha$ is viewed as a function on $\mR^n$ extended by zero, since $\supp\, (h^{\rm geo}_\alpha f)\subset U^{\rm geo}_\alpha$.
\end{defi}

\begin{rem} The spaces $H^s_p(M)$ generalize in a natural way the classical Sobolev spaces  $W^k_p(M)$,  $k\in \nat_0$, $1<p<\infty$, on Riemannian manifolds $M$: Let $
\Vert f\Vert_{W^k_p(M)}:=\sum_{l=0}^k \Vert \nabla^l f\Vert_{L_p(M)}, 
$ 
then $W^k_p(M)$ is the completion of $\mathcal{D}(M)$  in the  $W^k_p(M)$-norm,  cf. \cite{Aub1}, \cite{Aub2}. 
 As in the Euclidean case, on manifolds $M$ of bounded geometry one has the coincidence 
\beq\label{coinc}
W^k_p(M)=H^k_p(M), \qquad k\in \nat_0,\quad 1<p<\infty, 
\eeq
cf. \cite[Section~7.4.5]{Tri92}.\\
Alternatively, the fractional Sobolev spaces $H^s_p(M)$ on manifolds with bounded geometry can be characterized with the help of the Laplace-Beltrami operator, cf. \cite[Section 7.2.2 and Theorem 7.4.5]{Tri92}.  This approach was originally used by \cite{strich} and later on slightly modified  in \cite[Section  7.4.5]{Tri92} in the following way: 
 Let  $1<p<\infty$ and $\rho>0$. Let $s>0$, then $H^s_p(M)$ is the collection of all $f\in L_p(M)$ such that $f=(\rho\Id -\Delta)^{-s/2}h$ for some $h\in L_p(M)$, with the norm $\Vert f\Vert_{H^s_p(M)}=\Vert h\Vert_{L_p(M)}$. Let $s<0$, then $H^s_p(M)$ is the collection of all $f\in \mathcal{D}'(M)$ having the form $f=(\rho\Id-\Delta)^l h$ with $h\in {H}^{2l+s}_p(M)$, where $l\in \nat$ such that $2l+s>0$, and $\Vert f\Vert_{H^s_p(M)}=\Vert h\Vert_{{H}^{2l+s}_p(M)}$. Let $s=0$, then ${H}^0_p(M)=L_p(M)$. \\
In particular, the spaces $H^s_p(M)$ with $s<0$ are independent of the number $l$ appearing in their definition  in the sense of equivalent norms, cf. \cite[Definition 4.1]{strich}. The additional parameter $\rho>0$ used by {\sc Triebel} ensures that \eqref{coinc} also holds in this context as well. In particular, for $2\leq p<\infty$ one can choose $\rho=1$, cf. \cite[Rem. 1.4.5/1, p.~301]{Tri92}.  \\
Technically, it is possible to extend Definition \ref{H-geo} to the limiting cases when $p=1$ and $p=\infty$. However, already in the classical situation when $M=\rn$  the outcome is not satisfactory: the resulting spaces $H^s_p(\rn)$ have not enough Fourier multipliers, cf. \cite[p.~6, p.~13]{Tri92}, and there is no hope for a coincidence in the sense of \eqref{coinc}. Therefore, we restrict ourselves to $1<p<\infty$, but emphasize that the boundary cases are included in the outlook about $F$- and $B$-spaces in Section \ref{sec_TL}.
\end{rem}

\subsection{Sobolev norms on manifolds of bounded geometry using other trivializations}

For many applications the norm given in \eqref{geo_norm} is very useful. In particular, it enables us to  transfer many results known on $\mR^n$ to manifolds $M$ of bounded geometry. 
The choice of  geodesic coordinates, however, often turns out to be far too restrictive if one 
needs to adapt the underlying coordinates to a certain problem, e.g., to submanifolds $N$ of $M$ in order to study traces. Therefore,  in order to replace the geodesic trivializations in \eqref{geo_norm} we want to look for other \textit{'good'} trivializations which will result in equivalent norms (and hence yield the same spaces).

\begin{defi}\label{H-koord}
Let $(M^n,g)$ be a Riemannian manifold together with a uniformly locally finite trivialization $\uT=(U_\alpha, \kappa_{\alpha}, h_{\alpha})_{\alpha \in I}$. Furthermore, let $s\in \real$ and $1< p<\infty$.  Then the space $H^{s,\uT}_p(M)$ contains all distributions $f\in \mathcal{D}'(M)$ such that 
\begin{align*}
\Vert f\Vert_{H^{s,\uT}_{p}}:=\left(\sum_{\alpha\in I} \Vert (h_\alpha f)\circ \kappa_\alpha\Vert^p_{H^s_p(\mathbb{R}^n)}\right)^{\frac{1}{p}}
\end{align*}
is finite. Here again $(h_\alpha f)\circ \kappa_\alpha$ is viewed as function on $\R^n$, cf. \eqref{geo_norm} and below.  
\end{defi}
 
In general, the  spaces $H^{s,\uT}_p(M)$ do depend on the underlying trivialization $\uT$. One of our main aims will be to  investigate under which conditions on $\uT$ this norm is equivalent to the $H_p^s(M)$-norm. For that we will use the following terminology. 

\begin{defi}\label{bddcoord} Let $(M^n,g)$ be a Riemannian manifold of bounded geometry. Moreover, let a uniformly locally finite trivialization $\uT=(U_\alpha, \kappa_\alpha,h_\alpha)_{\alpha\in I}$ be given. We say that $\uT$ is admissible if the following conditions are fulfilled:
\begin{itemize}
 \item[(B1)] $\uA=(U_\alpha, \kappa_\alpha)_{\alpha\in I}$ is compatible with geodesic coordinates, i.e., for $\uA^{\rm geo}=(U^{\rm geo}_\beta, \kappa^{\rm geo}_\beta)_{\beta \in J}$ being a geodesic atlas of $M$ as in Example \ref{geod_coord_triv}
  there are constants $C_k>0$ for $k\in \mN_0$ such that for all $\alpha\in I$ and $\beta\in J$ with $U_{\alpha}\cap U^{\rm geo}_{\beta}\neq \varnothing$ and all $\a\in \mN_0^n$ with $|\a|\leq k$ 
  \[ |\uD^\a (\mu_{\alpha\beta}=(\kappa_\alpha)^{-1} \circ \kappa^{\rm geo}_\beta)|\leq C_k \qquad  \text{\ and\ } \qquad |\uD^\a (\mu_{\beta\alpha}=(\kappa^{\rm geo}_\beta)^{-1} \circ \kappa_\alpha)|\leq C_k.\]
\item[(B2)]  For all $k\in \mathbb{N}$ there exist $c_k>0$ such that for all $\alpha\in I$ and all multi-indices $\a$ with $|\a|\leq  k$ 
\[ |D^{\a}(h_\alpha\circ\kappa_\alpha)|\leq c_k.\]
\end{itemize}
\end{defi}

\begin{rem}\label{rem_comp_T} \hfill\\
{\bf i)} If (B1) is true for some geodesic atlas, it is true for any refined geodesic atlas. This follows immediately from Remark \ref{rem_bdgeom}.ii.\\
{\bf ii)} Condition (B1) implies in particular the compatibility of the charts in $\uT$ among themselves, i.e., for all $k\in \mN_0$ there are constants $C_k>0$ such that for all multi-indices $\a$ with $|\a|\leq k$ and all $\alpha,\beta\in I$ with $U_\alpha \cap U_\beta\neq \varnothing$ we have $|\uD^\a(\kappa_\alpha^{-1} \circ \kappa_\beta)|\leq C_k$. This is seen immediately when choosing $z \in U_\alpha\cap U_\beta$, considering the exponential map $\kappa_{z}^{\rm geo}$ around $z$,  applying the chain rule to $\uD^\a (\kappa_\alpha^{-1}\circ \kappa_\beta)=\uD^\a ((\kappa_\alpha^{-1}\circ \kappa_z^{\rm geo}) \circ ((\kappa_z^{\rm geo})^{-1}\circ \kappa_\beta))$. The same works for charts belonging to different admissible trivializations.
\end{rem}

\begin{thm}\label{indep_H} 
Let $(M,g)$ be a Riemannian manifold of bounded geometry, and  let $\uT=(U_{\alpha},\kappa_{\alpha}, h_\alpha)_{\alpha \in I}$ be an admissible trivialization of $M$. Furthermore, let $s\in \real$ and   ${1<p<\infty}$. Then,
\begin{align*}
H^{s,\uT}_{p}(M)=H^{s}_{p}(M), 
\end{align*}
i.e., for admissible trivializations of $M$ the resulting Sobolev spaces $H^{s,\uT}_{p}(M)$ do not depend on $\uT$.
\end{thm}

\begin{proof}
The proof is based on pointwise  multiplier assertions and diffeomorphism properties of the spaces $H^s_{p}(\rn)$, see Lemma \ref{Sob_Rn}. Let $\uT=(U_{\alpha},\kappa_{\alpha},h_{\alpha})_{\alpha\in I}$  be an admissible trivialization. Let a geodesic trivialization $\uT^{\rm geo}=(U^{\rm geo}_\beta, \kappa^{\rm geo}_\beta, h^{\rm geo}_\beta)_{\beta \in J}$ of $M$, see Example \ref{geod_triv}, be given. If $\alpha\in I$ is given, the index set $A(\alpha)$ collects all $\beta\in J$ for which $U_\alpha\cap U^{\rm geo}_\beta\neq \varnothing$. The cardinality of $A(\alpha)$ can be estimated from above by a constant independent of $\alpha$ since the covers are uniformly locally finite.

We assume $f\in H^{s}_{p}(M)$. By Lemma \ref{Sob_Rn} and Definition \ref{bddcoord} we have for all $\alpha\in I$
\begin{align*}
\Vert(h_{\alpha}f)\circ \kappa_{\alpha}\Vert_{H^s_{p}(\rn)} = &\left\Vert\sum_{\beta\in A(\alpha)} (h_{\alpha}h^{\rm geo}_\beta f)\circ \kappa_{\alpha}\right\Vert_{H^s_{p}(\rn)} \leq \sum_{\beta\in A(\alpha)} \left\Vert (h_{\alpha} h^{\rm geo}_\beta f)\circ \kappa_{\alpha}\right\Vert_{H^s_{p}(\rn)}\\
=& \sum_{\beta\in A(\alpha)} \left\Vert( h_\alpha h^{\rm geo}_{\beta}f)\circ (\kappa^{\rm geo}_{\beta}\circ (\kappa^{\rm geo}_{\beta})^{-1})\circ\kappa_{\alpha}\right\Vert_{H^s_{p}(\rn)}
\lesssim \sum_{\beta\in A(\alpha)} \left\Vert(h_\alpha h^{\rm geo}_{\beta}f)\circ \kappa^{\rm geo}_{\beta}\right\Vert_{H^s_{p}(\rn)}
\\
\lesssim& 
 \sum_{\beta\in A(\alpha)} \left\Vert(h^{\rm geo}_{\beta}f)\circ \kappa^{\rm geo}_{\beta}\right\Vert_{H^s_{p}(\rn)}.
\end{align*}
In particular, the involved constant can be chosen independently of $\alpha$. Then 
\[
\Vert f\Vert_{H^{s,\uT}_{p}(M)}=\left(\sum_{\alpha\in I} \Vert (h_{\alpha}f)\circ \kappa_{\alpha}\Vert^p_{H_p^s(\rn)}\right)^{1/p}
\lesssim \left(\sum_{\alpha\in I, \beta\in A(\alpha)} \Vert (h^{\rm geo}_{\beta}f)
\circ \kappa^{\rm geo}_{\beta}\Vert^p_{H_p^s(\rn)}\right)^{1/p}\lesssim \Vert f\Vert_{H^{s}_{p}(M)}  
\]
where the last estimate follows from { $\sum_{\alpha\in I,\, \beta\in A(\alpha)}=\sum_{\beta\in J,\, \alpha\in A(\beta)}$ and  the fact that  the covers are uniformly locally finite}. The reverse inequality is obtained analogously. Thus, $H^{s,\uT}_p(M)=H^s_p(M)$.  
\end{proof}

In view of Remark \ref{rem_bdgeom}.iii, we would like to have a similar result for trivializations satisfying condition (B1).

\begin{lem} Let $(M,g)$ be a Riemannian manifold with positive injectivity radius, and let $\uT=(U_\alpha, \kappa_\alpha, h_\alpha)_{\alpha\in I}$ be a uniformly locally finite trivialization. Let  $g_{ij}$ be the coefficient matrix of $g$ and  $g^{ij}$ its inverse with respect to the coordinates $\kappa_\alpha$. Then,  $(M,g)$ is of bounded geometry and $\uT$ fulfills (B1) if, and only if, the following is fulfilled: 

For all $k\in \mN_0$ there is a constant $C_k>0$ such that for all multi-indices $\a$ with $|\a|\leq k$, 
\begin{equation}\label{B1-equiv} |\uD^\a g_{ij}|\leq C_k \quad \text{\ and\ }\quad  |\uD^\a g^{ij}|\leq C_k \end{equation}
holds in all charts $\kappa_\alpha$.
\end{lem}

\begin{proof} Let \eqref{B1-equiv} be fulfilled. Then, $(M,g)$ is of bounded geometry since $R^M$ in local coordinates  is given by a polynomial in $g_{ij}$, $g^{ij}$ and its derivatives.  Moreover,  condition (B1) follows from \cite[Lemma 3.8]{Schick01} -- we shortly sketch the argument here: Let $\Gamma_{ij}^k$ denote the Christoffel symbols with respect to coordinates $\kappa_\alpha$ for $\alpha\in I$. By \eqref{Christ_coord} and \eqref{B1-equiv}, there are constants $C_k>0$ for $k\in \mN_0$ such that $|\uD^\a \Gamma_{ij}^k|\leq C_k$ for all $\alpha\in I$ and all $\a\in\mN_0^n$ with $|\a|\leq k$.
Moreover, fix $r>0$ smaller than the injectivity radius of $M$. Let $\uA^{\rm geo}=(U^{\rm geo}_\beta=B_r(p_\beta^{\rm geo}), \kappa_\beta^{\rm geo})_{\beta\in J}$ be a geodesic atlas of $M$ where $r>0$ is smaller than the injectivity radius. {We get  that $(\kappa_\alpha)^{-1}\circ \kappa_\beta^{\rm geo}(x)= \Phi(1,\kappa_\alpha^{-1}(p_\beta), \kappa_\alpha^*(\lambda_\beta(x)))$ where $\Phi$ is the geodesic flow. Then, together with Lemma \ref{flow_lem} it follows that $(\kappa_\alpha)^{-1}\circ \kappa_\beta^{\rm geo}$ and all its derivatives are uniformly bounded independent on $\alpha$ and $\beta$. Moreover, note that $(\kappa_\beta^{\rm geo})^{-1}\circ \kappa_\alpha: \kappa_\alpha^{-1}(U_\alpha\cap U_\beta^{\rm geo})\subset B_r^n\to (\kappa_\beta^{\rm geo})^{-1}(U_\alpha\cap U_\beta^{\rm geo})\subset B_r^n$ is bounded by $r$. Hence, together with the chain rule applied to $((\kappa_\beta^{\rm geo})^{-1}\circ \kappa_\alpha) \circ ((\kappa_\alpha)^{-1}\circ \kappa_\beta^{\rm geo})  =\Id$ condition (B1)  
follows for all $(\alpha,  \beta)$.}

Conversely, let $(M,g)$ be of bounded geometry, and let condition (B1) be fulfilled. Then, by Remark \ref{rem_bdgeom}.iii and the transformation formula \eqref{trans_g} for $\alpha\in I$ and $\beta\in J$,  condition \eqref{B1-equiv} follows.
\end{proof}

\subsection{Besov spaces on manifolds}
Similar to the situation on $\rn$ we can define Besov spaces on manifolds via real interpolation of fractional Sobolev spaces $H^s_p(M)$. 

\begin{defi}\label{def-B-spaces}
Let $(M,g)$ be a manifold of bounded geometry. 
Furthermore, let $s_0,s_1\in \real$,  $1<p<\infty$ and $0<\Theta<1$. We define   
\beq\label{def-B-int}
B^{s}_{p,p}(M):=\left(H^{s_0}_p(M), H^{s_1}_p(M)\right)_{\Theta,p},
\eeq
where $s=\Theta s_0+(1-\Theta)s_1$.
\end{defi}

\begin{rem}\label{rem-B} The fractional Sobolev spaces $H^{s_i}_p(M)$ appearing  in Definition \ref{def-B-spaces} above should be understood in the sense of Definition \ref{H-koord}. For the sake of simplicity we restrict ourselves to admissible trivializations $\uT$  when defining Besov spaces on $M$. This way, 
 by Theorem \ref{indep_H}, we can omit the dependency on the trivializations $\uT$ from our notations in \ref{def-B-int} since resulting norms are equivalent and  yield the same spaces. Note that our spaces are well-defined since \eqref{def-B-int} is actually independent of $s_0$ and $s_1$.  An explanation is given in \cite[Theorem~7.3.1]{Tri92}.
Furthermore, an equivalent norm for $f\in B^s_{p,p}(M)$ is given by 
\beq\label{B-def}
\Vert f\Vert_{B^{s}_{p,p}(M)}=\left(\sum_{\alpha\in I} \Vert (h_\alpha f)\circ \kappa_\alpha\Vert^p_{B^s_{p,p}(\mathbb{R}^n)}\right)^{\frac{1}{p}}.
\eeq
 We sketch the proof.  {By $\ell_p(\Hsp)$ we denote the sequence space containing all sequences $\{f_{\alpha}\}_{\alpha\in I}$ such that the norm
\[
\|f_{\alpha}\|_{\ell_p(\Hsp)}:=\left(\sum_{\alpha\in I}\|f_{\alpha}\|_{\Hsp}^p\right)^{\frac 1p}
\]
is finite,} 
similar for $\ell_p(B^s_{p,p})$ with obvious modifications. 
Let $A(\alpha)=\{\beta\in I\ |\  U_{\beta}\cap U_{\alpha}\neq \emptyset\} $, and let 
$\Lambda_{\alpha}=\left(\sum_{\beta \in A(\alpha)}h_{\beta}\right)\circ \kappa_{\alpha}$. 
We define a linear and bounded operator
\[
\Lambda: \ell_p(\Hsp(\rn))\longrightarrow \Hsp(M), 
\]
via 
\[
\Lambda\left\{f_{\beta}\right\}_{\beta\in I}=\sum_{\beta\in I}(\Lambda_{\beta}f_{\beta})\circ \kappa_{\beta}^{-1},
\]
where $(\Lambda_{\beta}f_{\beta})\circ \kappa_{\beta}^{-1}$ is extended outside $U_{\beta}$ by zero. Furthermore, we consider
\[
\Psi: \Hsp(M)\longrightarrow \ell_p(\Hsp(\rn)),
\]
given by 
\[
\Psi (f)=\{(h_{\alpha}f)\circ \kappa_{\alpha}\}_{\alpha\in I}
\]
which is also a linear and bounded operator. In particular, we have that 
\[
\Lambda \circ \Psi=\Id \qquad \text{(identity in \ }\Hsp(M)).
\]
Having arrived at a standard situation of interpolation theory we use the method of retraction/coretraction, cf. \cite[Theorem~1.2.4]{T-interpol}, reducing \eqref{B-def} to the question whether 
\beq\label{b-int}
\Big(\ell_p(H^{s_0}_p), \ell_p(H^{s_1}_p)\Big)_{\Theta,p}
=\ell_p\Big(\big({H^{s_0}_p},H^{s_1}_p\big)_{\Theta,p}\Big),
\eeq
for $1<p<\infty$, $s_0, s_1\in \real$, $0<\Theta<1$, and $s=\Theta s_0+(1-\Theta)s_1$, 
which can be found in \cite[Theorem~1.18.1]{T-interpol}. Since by definition of Besov spaces  the right hand side of \eqref{b-int}  coincides with $\ell_p(B^s_{p,p})$, this proves \eqref{B-def}. 

\end{rem}

\section{Coordinates on submanifolds and Trace Theorems}

From now on let $N^k\subset M^n$  be an embedded submanifold, meaning, there is a $k$-dimensional manifold $N'$ and an injective immersion $f: N'\to M$  with $f(N')=N$. The aim of this section is to prove a Trace Theorem for $M$ and $N$. We restrict ourselves to submanifolds of bounded geometry in the following sense:

\begin{defi}\label{bdd_geo} 
Let $(M^n,g)$ be a  Riemannian manifold with a $k$-dimensional  embedded submanifold $(N^k, g|_N)$. We say
that $(M,N)$ is of bounded geometry if the following is fulfilled
\begin{itemize}
\item[(i)] $(M,g)$ is of bounded geometry.
\item[(ii)] The injectivity radius $r_N$ of $(N, g|_N)$ is positive.
 \item[(iii)] 

 There is a collar around $N$ (a tubular neighbourhood of fixed radius), i.e., there is $r_\partial>0$ such that for all $x,y\in N$ 
with $x\neq y$ the normal balls $B_{r_\partial}^\perp (x)$ and $B_{r_\partial}^\perp (y)$ are disjoint where \[ B_{r_\partial}^\perp (x):=\{z\in M\ |\ dist_M (x,z)\leq {r_\partial}, \exists \epsilon_0 \forall \epsilon<\epsilon_0:\  dist_M (x,z)=dist_M (B^N_\epsilon (x),z)\}\] 
 \begin{minipage}{0.45\textwidth}
 with 
$$B^N_\epsilon (x)=\{ u\in N\ |\ dist_N(u,x)\leq \epsilon\} $$ {and   
$dist_{M}$  and $dist_N$ denote the distance functions in $M$ and $N$, respectively}.
\end{minipage}\hfill 
\begin{minipage}{0.4\textwidth}
\begin{psfrags}
\psfrag{x}{\small $x$}
\psfrag{y}{\small $y$}
\psfrag{A}{\small $B_{r_\partial}^\perp (x)$}
\psfrag{B}{\small $B_{r_\partial}^\perp (y)$}
\psfrag{N}{$N$}
\psfrag{C}{\small{$B^N_\epsilon (x)$} }
{\includegraphics[width=5cm]{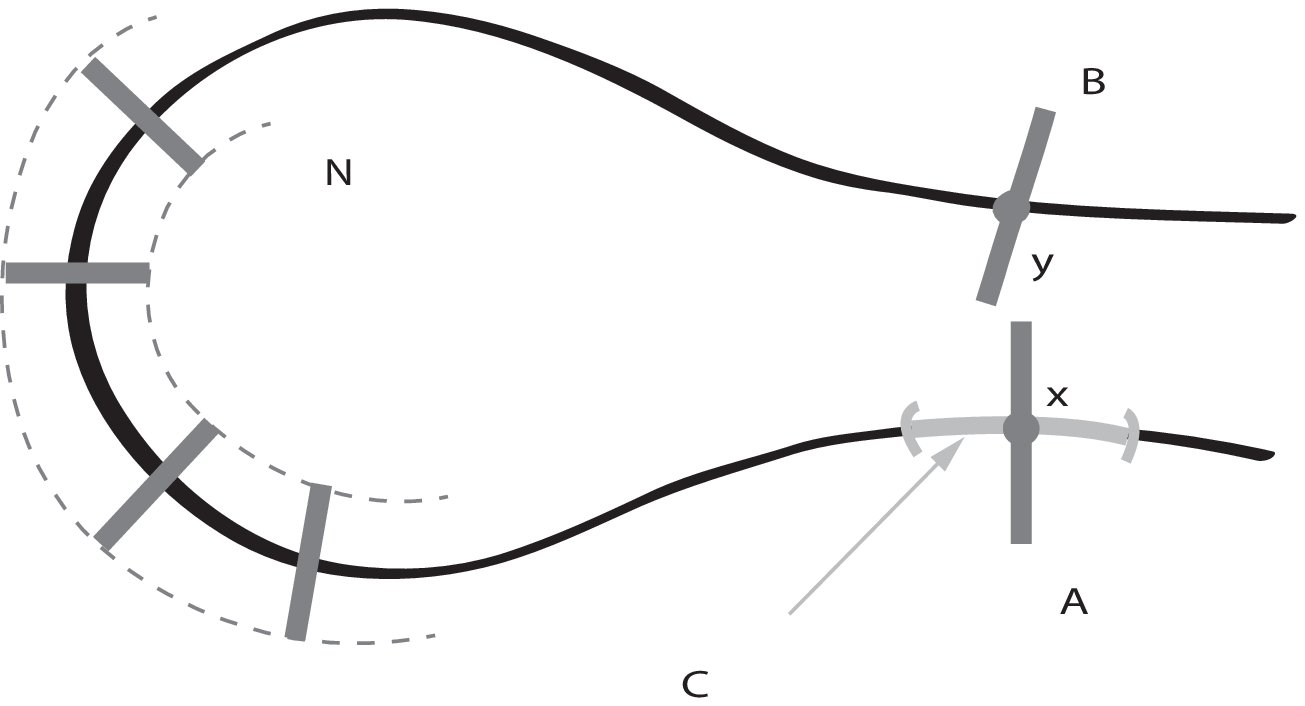}}
\end{psfrags}
\end{minipage}
 \item[(iv)] The mean curvature $l$ of $N$ given by
\[ l(X,Y):= \nabla^M_X Y- \nabla^N_X Y \quad \text{for all}\ X,Y\in T N,\]
and all its covariant derivatives are bounded. Here, $\nabla^M$ is the Levi-Civita connection of $(M,g)$ and $ \nabla^N$ the one of $(N, g|_N)$.
\end{itemize}
\end{defi}

\begin{rem}\hfill\label{rem-Fermi-collar}\\[-0.5cm]
\begin{itemize}
\item[i)] If the normal bundle of $N$ in $M$ is trivial, condition (iii) in Definition \ref{bdd_geo} simply means that $\{z\in M\ |\ dist_M(z,N)\leq r_\partial\}$ is diffeomorphic to $B_{r_\partial}^{n-k}\times N$. Then \[F: B_{r_\partial}^{n-k}\times N \to M;\ (t,z)\mapsto \exp^M_z \left(t^i\nu_i\right)\]
is a diffeomorphism onto its image, where $(t^1, ...,t^{n-k})$ are the coordinates for $t$ with respect to a standard orthonormal basis on $\mR^{n-k}$ and $(\nu_1, \ldots, \nu_{n-k})$ is an orthonormal frame for the normal bundle of $N$ in $M$.

If the normal bundle is not trivial (e.g. consider a noncontractible circle $N$ in the infinite M\"obius strip $M$), $F$ still exists locally, which means that for all $x\in N$ and $\epsilon$ smaller than the injectivity radius of $N$, the map $F: B_{r_\partial}^{n-k}\times B_\epsilon^N(x) \to M;\ (t,z)\mapsto \exp^M_z \left(t^i\nu_i\right)$ is a diffeomorphism onto its image. All included quantities are as in the case of a trivial vector bundle, but $\nu_i$ is now just a local orthonormal frame of the normal bundle. By abuse of notation, we suppress here and in the following the dependence of $F$ on $\epsilon$ and $x$.
\item[ii)] The illustration below on the left hand side shows a submanifold $N$ of a manifold $M$ that admits a collar. \\ On the right hand side one sees that for $M=\real^2$ the submanifold $N$ describing the curve which for large enough $x$ contains the graph of $x\mapsto x^{-1}$ together with the $x$-axes does not have a collar. This situation is therefore excluded by Definition \ref{bdd_geo}. However,  to a certain extend, manifolds as in the  picture on the right hand side can still be treated, cf. Example \ref{ex-loccoll} and Remark \ref{rem_th_loccoll}.\\[0.2cm]
\begin{minipage}{0.45\textwidth}
\begin{psfrags}
\psfrag{x}{$x$}
\psfrag{y}{$y$}
\psfrag{z}{$z$}
\psfrag{e}{$\exp_z^{M}$}
\psfrag{N}{$N$}
\psfrag{M}{$M$}
\psfrag{C}{\small{$F(B_{r_{\delta}}^{n-k}\times N)$} }
{\includegraphics[width=6.1cm]{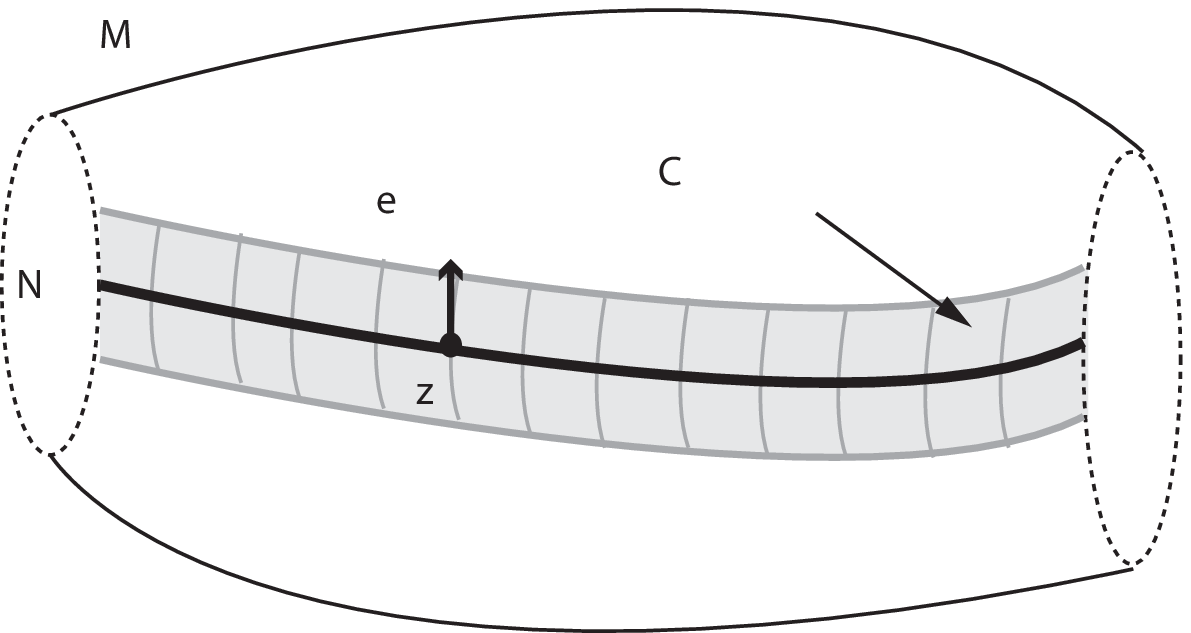}}
\end{psfrags}\\
\end{minipage}\hfill \begin{minipage}{0.45\textwidth}
\begin{psfrags}
\psfrag{x}{$x$}
\psfrag{y}{$y$}
\psfrag{z}{$z$}
\psfrag{a}{$x^{-1}$}
\psfrag{N}{$N$}
\psfrag{M}{$M=\real^2$}
\psfrag{g}{$\gamma$}
{\includegraphics[width=5.5cm]{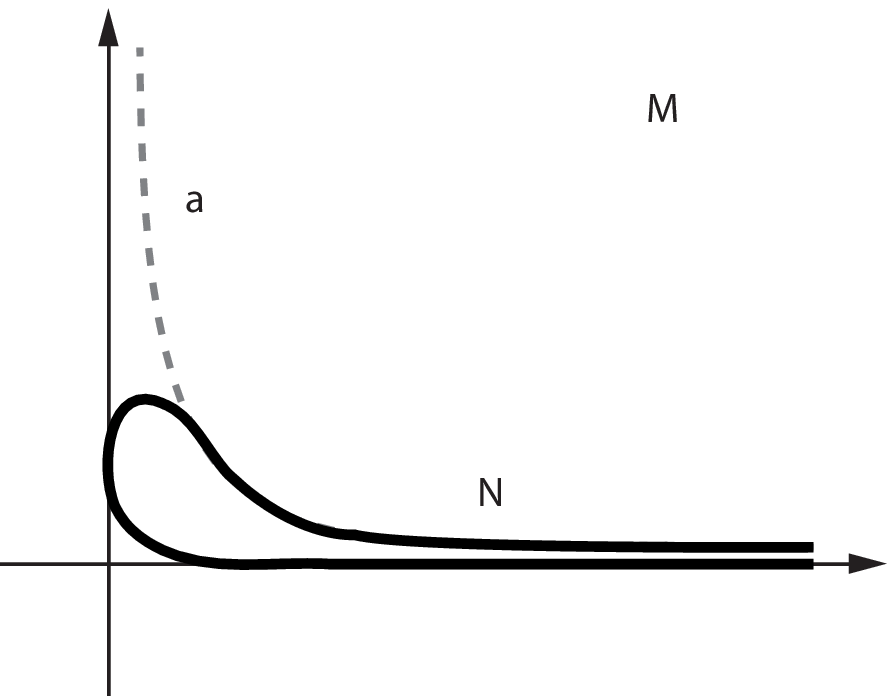}}
\end{psfrags}\\
\end{minipage}
\item[iii)] Although our notation $(M,N)$ hides the underlying metric $g$, this is obviously part of the definition and fixed when talking about $M$.
\item[iv)] If $N$ is the boundary of the manifold $M$, the counterpart of Definition \ref{bdd_geo} can be found in \cite[Definition 2.2]{Schick01}, where also Fermi coordinates are introduced and certain properties  discussed. In Section \ref{FC_sec}, we adapt some of the methods from \cite{Schick01} to our situation. Note that the normal bundle of the boundary of a manifold is always trivial, which explains why in \cite[Definition 2.2]{Schick01} condition (iii) of Definition \ref{bdd_geo}  reads as in Remark \ref{rem-Fermi-collar}.i.
\end{itemize}
\end{rem}

\subsection{Fermi coordinates}\label{FC_sec}

In this subsection we will introduce Fermi coordinates, which are special coordinates adapted to a submanifold $N$ of $M$ where $(M,N)$ is of bounded geometry. The resulting trivialization is used to prove the Trace Theorem in Section \ref{trace_th_sec}. 

\begin{defi}[\bf Fermi coordinates] \label{FC}  We use the notations from Definition \ref{bdd_geo}. Let $(M^n,N^k)$ be of bounded geometry. Let $R=\min\left\{ \frac{1}{2}r_N, \frac{1}{4}r_M, \frac{1}{2}r_\partial\right\}$, where $r_N$ is the injectivity radius of $N$ and $r_M$ the one of $M$. Let there be countable index sets $I_N\subset I$ and sets of points $\{p^N_\alpha\}_{\alpha\in I_N}$ and $\{p_\beta\}_{\beta\in I\setminus I_N}$  in $N$ and $M\setminus U_R(N)$, respectively, where $U_{R}(N):=\cup_{x\in N} B^{\perp}_R(x)$. Those sets are chosen such that
\begin{itemize}
 \item[(i)]  The collection of the metric balls $(B_{R}^N(p^N_\alpha))_{\alpha\in I_N}$ gives a uniformly locally finite cover of $N$.
 Here the balls are meant to be metric with respect to the induced metric $g|_N$.
 \item[(ii)]  The collection of metric balls $(B_{R}(p_\beta))_{\beta\in I\setminus I_N}$ covers $M\setminus U_{R}(N)$ and is uniformly locally finite on all of $M$.
\end{itemize}
We consider the covering $(U_\gamma)_{\gamma \in I}$ with $U_\gamma=B_{R}(p_\gamma)$ for $\gamma\in I\setminus I_N$ and $U_\gamma=U_{p_\gamma^N}:=F(B_{2R}^{n-k}\times B^N_{2R} (p_\gamma^N))$ with $\gamma\in I_N$.  Coordinates on $U_\gamma$ are chosen to be geodesic normal coordinates around $p_\gamma$ for $\gamma\in I\setminus I_N$. Otherwise, if $\gamma\in I_N$, coordinates are given by Fermi coordinates 
\begin{equation}\label{double_flow} \kappa_\gamma: V_{p_\gamma^N}:=B_{2R}^{n-k}\times B^k_{2R} \to U_{p_{\gamma}^N},\quad  (t,x)\mapsto \exp^M_{\exp^N_{p_{\gamma}^N}(\lambda_\gamma^N (x))} \left( t^i\nu_i \right)\end{equation}
where $(t^1,\ldots, t^{n-k})$ are the coordinates for $t$ with respect to a standard orthonormal basis on $\mR^{n-k}$, $(\nu_1, \ldots, \nu_{n-k})$ is an orthonormal frame for the normal bundle of $B_{2R}^N(p_\gamma^N)$ in $M$, $\exp^N$ is the exponential map on $N$ with respect to the induced metric $g|_N$, and $\lambda_\gamma^N: \mR^k\to T_{p_\gamma^N}N$ is the choice of an orthonormal frame on $T_{p_\gamma^N}N$.
\end{defi}
\begin{center}
\begin{psfrags}
    \psfrag{x}{$\real^k$}
	\psfrag{y}{$\real^{n-k}$}
	\psfrag{z}{\small{$p_{\gamma}^N$}}
    \psfrag{e}{$\exp_z^{M}$}	
     \psfrag{V}{$V_{p_{\gamma}^N}$}	
    \psfrag{M}{$M$}
    \psfrag{N}{$N$}
    \psfrag{a}{\small{\small $2R$}}
    \psfrag{P}{\small{\small $(x,t)$}}
    \psfrag{k}{$\kappa_{\gamma}$}
    \psfrag{U}{$U_{p_{\gamma}^N}$}
    \psfrag{c}{\small $\exp^N_{p_{\gamma}^N}(\lambda_{\gamma}^N(x))$}
    \psfrag{d}{\small $\exp^M_{\exp^N_{p_{\gamma}^N}(\lambda_{\gamma}^N(x))}(t^i\nu_i)$}
	{\includegraphics[width=14.5cm]{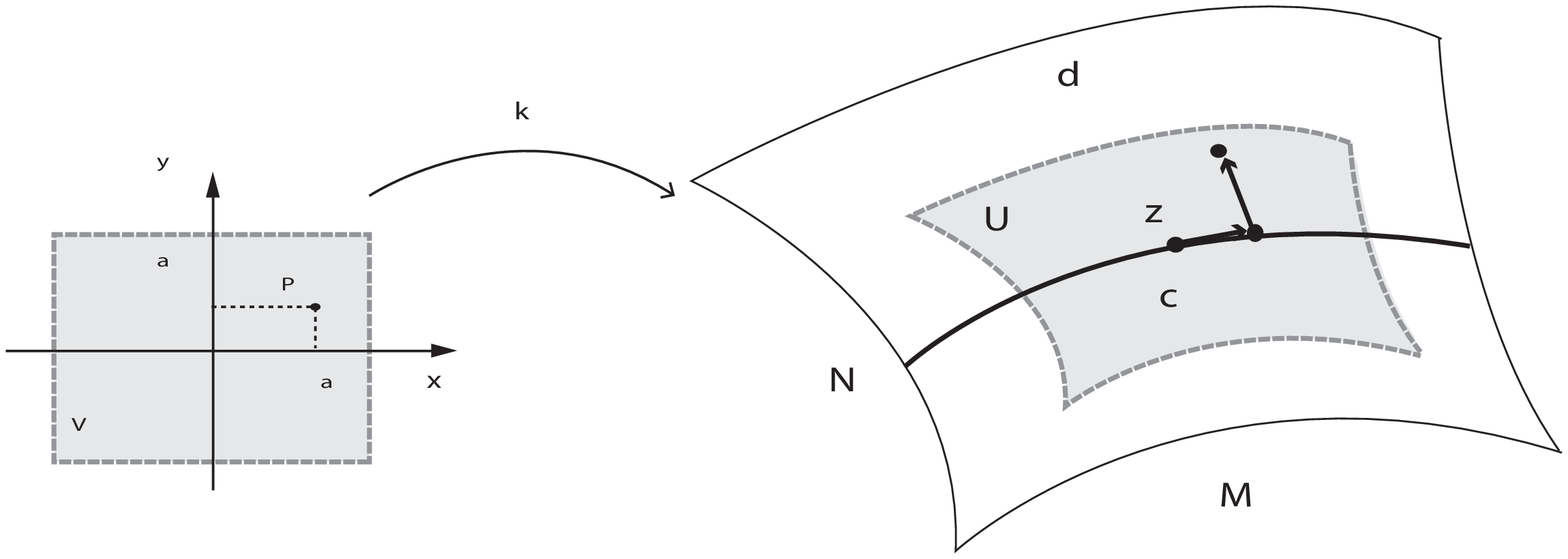}}
\end{psfrags}
\end{center}

Before giving a remark on the existence of the points $\{p_\gamma\}_{\gamma\in I}$ claimed in the Definition above, we prove two lemmata.

\begin{lem}\label{FC-0-bound} Let $(M^n, N^k)$ be of bounded geometry, and let $C>0$ be such that the Riemannian curvature tensor fulfills $|R^M|\leq C$ and mean curvature of $N$ $|l|\leq C$.
Fix $z\in N$ and $R$ as in Definition \ref{FC}. Let $U=F(B_{2R}^{n-k}\times B_{2R}^N(z))$, and let a chart $\kappa$ for $U$ be defined as above.  Then there is a constant $C'>0$ only depending on $C$, $n$ and $k$, such that $|g_{ij}|\leq C'$ and $|g^{ij}|\leq C'$ where $g_{ij}$ denotes the metric $g$ with respect to $\kappa$.
\end{lem}

\begin{proof} For $N$ being the boundary of $M$ this was shown in \cite[Lemma 2.6]{Schick01}. We follow the idea given there and use the extension of the Rauch comparison theorem to submanifolds of arbitrary codimension given by Warner in \cite[Theorem 4.4]{War}. For the comparison, let $M_C$ and $M_{-C}$ be two complete $n$-dimensional Riemannian manifolds  of constant sectional curvature $C$ and $-C$, respectively. In each of them we choose a $k$-dimensional submanifold $N_C$ and $N_{-C}$, points $p_{\pm C}\in N_{\pm C}$ and a chart of $M_{\pm C}$ around $p_{\pm C}$ given by Fermi coordinates such that all eigenvalues of the second fundamental form with respect to those coordinates at $p_{\pm C}$ are given by $\pm C$ (this is always possible, cf. \cite[Chapter 7]{Spiv4}). 
Let $(\nu_i)_{1\leq i\leq n-k}$ be an orthonormal frame of the normal bundle of $U\cap N$ and $(e_i)_{1\leq i\leq k}$ be an orthonormal frame of $T|_{U\cap N} N$ { obtained via geodesic flow on $N$}. Let the frame $(\nu_1, \ldots, \nu_{n-k}, e_1, ..., e_k)$ be transported to all of $U$ via parallel transport along geodesics normal to $N$ -- the transported vectors are also denoted by $\nu_i$ and $e_i$, respectively.

Then, we are in the situation to apply \cite[Theorem 4.4]{War}: Let now $p\in U$ and $v\in T_pU$ with $v\perp \nu_i$ for all $1\leq i\leq n-k$. Then, the  comparison theorem yields  constants $C_1,C_2>0$, depending only on $C, n$, and $k$,  such that $ C_1|v|_E^2 \leq g_p(v,v)\leq C_2|v|_E^2$, { where $|.|_E$ denotes the Euclidean metric with respect to the basis $(e_i)$}. Moreover, we have $g_p(\nu_i,\nu_j)=\delta_{ij}$ and $g_p(\nu_i, e_l)=0$ for all $1\leq i,j\leq n-k$ and $1\leq l\leq k$, since this is true for $p\in U\cap N$, and this property is preserved by parallel transport. 
Altogether this implies the claim.
\end{proof}

The previous lemma enables us to show that  $(N, g|_N)$ is also of bounded geometry.

\begin{lem}\label{N_bdd_geo} 
If $(M,N)$ is of bounded geometry, then $(N, g|_N)$ is  of bounded geometry.
\end{lem} 

\begin{proof} Since Definition \ref{bdd_geo} already includes the positivity of the injectivity radius of $N$, it is enough to show that $(\nabla^N)^k R^N$, where $R^N$ is the Riemannian curvature of $(N, g|_N)$,  is bounded for all $k\in \mathbb{N}_0$:

{Let $z\in N$. We consider geodesic normal coordinates  $\kappa^{\rm geo}:  B_{2R}^k \to  U^{\rm geo}=B_{2R}^N(z)$ on $N$ around $z$ and Fermi coordinates  $\kappa: B_{2R}^{n-k}\times B_{2R}^k \to U=F(B_{2R}^{n-k}\times B_{2R}^N(z)$ on $M$ around $z$, cf. Definition \ref{FC}.} Let $g_{ij}$ be the metric with respect to the coordinates given by $\kappa$, and let $g^{ij}$ be its inverse. 
Since $M$ is of bounded geometry,  Lemma \ref{FC-0-bound} yields  a constant $C$ independent on $z$ such that we have  $|g_{ij}|\leq C$ and $|g^{ij}|\leq C$.
Together with the uniform boundedness of  $R^M$, $l$ and their covariant derivatives, we obtain that their representations $R^M_{ijkl}$, $l_{rs}$ and their derivatives in the coordinates given by $\kappa$  are uniformly bounded for all $1\leq i,j,k,l\leq n$, $n-k+1\leq r,s\leq n-k$.  
Then the claim follows by the Gauss' equation \cite[p. 47]{Spiv4},
\begin{equation*}
 g(R^N(U,V)W,Z)=g(R^M(U,V)W,Z)+g(l(U,Z),l(V,W))-g(l(U,W),l(V,Z)) \; \  \text{for\ all\ } U,V,W,Z\in TN,
\end{equation*}
and the formulas for covariant derivatives of tensors along $N$. We refer to \cite[Lemma 2.22]{Schick01}, where everything is stated for hypersurfaces but the formulas remain true  for arbitrary codimension subject to obvious modifications.
\end{proof}

\begin{rem} \label{rem_FC}
{\bf i)} By construction the covering $(U_\gamma)_{\gamma\in I}$ is uniformly locally finite and $(U_\gamma':=U_\gamma\cap N, \kappa_\gamma'=\kappa_\gamma|_{\kappa_{\gamma}^{-1}(U_\gamma')})_{\gamma\in I_N}$ gives a geodesic atlas on $N$. Moreover, none of the balls $B_R(p_\beta)$ with $\beta\in I\setminus I_N$ intersects $N$.\\
{\bf ii)} Existence of points $p_\alpha^N$ and $p_\beta$ as claimed in Definition \ref{FC} {(note that the proofs of Lemmas  \ref{FC-0-bound} and \ref{N_bdd_geo} only use the definition of Fermi coordinates on a single chart and the existence of the points $p_\alpha^N$ and $p_\beta$)}: We choose a maximal set   $\{p_\alpha^N\}_{\alpha\in I_N}$ of points in $N$ such that the metric balls $B^N_{\frac{R}{2}}(p_\alpha^N)$ are pairwise disjoint. Then, the balls $B^N_{R}(p_\alpha^N)$ cover $N$. Since by Lemma \ref{N_bdd_geo} the submanifold $(N,g|_N)$ is of bounded geometry, the volume of metric balls in $N$ with fixed radius is uniformly bounded from above and from below away from zero. Let a ball $B_{2R}^N(p_\alpha^N)$ be intersected by $L$ balls $B^N_{2R}(p_{\alpha'}^N)$. Then the union of the $L$ balls $B^N_{2R}(p_{\alpha'}^N)$  forms a subset of $B_{4R}^N(p_\alpha^N)$. Comparison of the volumes gives an upper bound on $L$. Hence, the balls in $(B_{2R}^N(p_\alpha^N))_{\alpha\in I_N}$ cover $N$ uniformly 
locally finite. Moreover, choose a maximal set of points $\{p_\beta\}_{\beta\in I\setminus I_N}\subset M\setminus U_R(N)$ such that the metric 
balls $B_
{\frac{R}{2}}(p_\beta)$ are pairwise disjoint in $M$. Then the 
balls $B_{R}(p_
\beta)$ cover $M\setminus U_R(N)$. Trivially the balls $B_{R}(p_\beta)$ for $\beta\in I\setminus I_N$ cover $\cup B_{R}(p_\beta)$, and by volume comparison as above this cover is uniformly locally finite.
\end{rem}

\begin{lem}\label{FC-1}
 The atlas $(U_\gamma, \kappa_\gamma)_{\gamma\in I}$ introduced in Definition \ref{FC} 
fulfills condition (B1).
\end{lem}

\begin{proof}For all $\gamma\in I\setminus I_N$ the chart $\kappa_\gamma$ is given by geodesic normal coordinates and, thus, condition (B1) follows from Remark {\ref{rem_bdgeom}.ii.}

Let now $\gamma\in I_N$.  Then the claim follows from \cite[Lemma 3.9]{Schick01}. We sketch the proof.  Consider a chart $(B_{4R}(p_\alpha^N), \kappa^{\rm geo})$ in $M$  and a chart $(B_{2R}^N(p_\alpha^N), \kappa^{N,\rm geo})$ in $(N,g|_N)$ both given by geodesic normal coordinates around $p_\alpha^N$ for $\alpha\in I_N$. Note that $4R< r_M$ by Definition \ref{FC}.  

Let $\Phi_2$ be the geodesic flow in $(N,g|_N)$ with respect to the coordinates given by $\kappa^{N, \rm geo}$, cf. Example \ref{geo_flow}. Let $\Phi_1$ be the corresponding geodesic flow in $(M,g)$ given by $\kappa^{\rm geo}$. {Then, $\Phi_2(1,0,x)=(\kappa^{N, \rm geo})^{-1}\circ \exp_{p_\alpha^N}^N( \lambda^N_\alpha (x))$ and
by \eqref{double_flow}, $\kappa_\alpha(t,x)= \kappa^{\rm geo}\circ \Phi_1(1,\Phi_2(1,0,x), (\kappa^{\rm geo})^*(t^i\nu_i))$ with $t=(t^1,\ldots, t^{n-k})\in \mR^{n-k}$.} Since $(M,g)$ is of bounded geometry, the coefficient matrix $g_{ij}$ of $g$ with respect to $\kappa^{\rm geo}$, its inverse and all its derivatives are uniformly bounded by \eqref{comp-coord}. Moreover, by Lemma \ref{N_bdd_geo} $(N,g|_N)$ is also of bounded geometry and, thus, we get an analogous statement for the coefficient matrix of $g|_N$ with respect to $\kappa^{N, \rm geo}$. Hence, applying Lemma \ref{flow_lem} to the differential equation of the geodesic flows, see Example \ref{geo_flow} and \eqref{Christ_coord}, we obtain that $(\kappa^{\rm geo})^{-1}\circ \kappa_\alpha$ and all its derivatives are bounded independent on $\alpha$. Conversely, $(\kappa_\alpha)^{-1}\circ \kappa^{\rm geo}: (\kappa^{\rm geo})^{-1}\circ \kappa_\alpha (B_{2R}^{n-k}\times B_{2R}^k) \subset B_{4R}^n\to B_{2R}^{n-k}\times B_{
2R}^k$ is bounded independent on $\alpha$. Hence, by using the chain rule on $((\kappa_\alpha)^{-1}\circ \kappa^{\rm geo}) \circ ((\kappa^{\rm geo})^{-1}\circ \kappa_\alpha) =\Id $ one sees that also the derivatives of  $(\kappa_\alpha)^{-1}\circ \kappa^{\rm geo}$ are  uniformly bounded, which gives the claim.
\end{proof}

\begin{lem}\label{part_unit_FC}
There is a partition of unity subordinated to the Fermi coordinates introduced in Definition \ref{FC} fulfilling condition (B2).
\end{lem}

\begin{proof} By Lemma \ref{N_bdd_geo}, $(N, g|_N)$ is of bounded geometry. Then,
by Example {\ref{geod_triv}}, there is a partition of unity $h'_\alpha$ subordinated to a geodesic atlas $(U_\alpha':=U_\alpha\cap N=B_{2R}^N(p_\alpha^N), \kappa_\alpha'=\kappa_\alpha|_{\kappa_\alpha^{-1}(U'_\alpha)})_{\alpha\in I_N}$ of $N$ such that for each $\a\in\mN_0^k$ the derivatives $\uD^\a (h'_\alpha\circ \kappa'_\alpha)$ are uniformly bounded independent of $\alpha$. Since by construction the balls $B_R^N(p_\alpha^N)$ already cover $N$ the functions $h'_\alpha$ can be chosen such that $\supp\, h'_\alpha\subset B_R^N(p_\alpha^N)$.

Choose a function $\psi: \mR^{n-k}\to [0,1]$ that is compactly supported on $B_{\frac 32 R}^{n-k}\subset \mR^{n-k}$ and $\psi|_{ B_{R}^{n-k}}=1$. Set $h_\alpha = (\psi\times (h'_\alpha\circ \kappa'_\alpha))\circ\kappa_\alpha^{-1}$ { on $U_\alpha$ and zero outside}. Then, $\supp\, h_\alpha\subset U_\alpha$ and all $\uD^\a (h_\alpha \circ \kappa_\alpha)$ are uniformly bounded by a constant depending on $|\a|$ but not on $\alpha\in I_N$.

Let $S\subset M$ be a maximal set of points containing the set $\{p_\beta\}_{\beta \in I\setminus I_N}$ of Definition \ref{FC} such that the metric balls in $\{B_{\frac{R}{2}}(p)\}_{p\in S}$ are pairwise disjoint. Then $(B_R(p))_{p\in S}$ forms a uniformly locally finite cover of M. We equip this cover with a geodesic trivialization $(B_R(p), \kappa^{\rm geo}_p, h^{\rm geo}_p)_{p\in S}$, see Example \ref{geod_triv}. For $\beta\in I\setminus I_N$ we have by construction $\kappa_{\beta}=\kappa^{\rm geo}_{p_\beta}$ and set 
\[h_{\beta}=\begin{cases} (1-\sum_{\alpha\in I_N} h_\alpha)\frac{h^{\rm geo}_{p_\beta}}{\sum_{\beta'\in I\setminus I_N} h^{\rm geo}_{p_{\beta'}}}, & \mathrm{where\ } \quad  \sum_{\beta'\in I\setminus I_N}h^{\rm geo}_{p_\beta'}\neq 0\\
                     0, & \mathrm{else}.
                    \end{cases}
\] 
Next, we will argue that all $h_\beta$ are smooth: It suffices to prove the smoothness in points $x\in M$ on the boundary of $\{\sum_{\beta'\in I\setminus I_N}h^{\rm geo}_{p_\beta'}\neq 0\}$. For all other $x$ smoothness follows by smoothness of the functions $h_\alpha$ and $h_p^{\rm geo}$. Let now $x\in M$ as specified above. Then  $\sum_{\beta'\in I\setminus I_N} h^{\rm geo}_{p_{\beta'}}(x)= 0$ and, thus, $x\in U_R(N)$ (cf. Remark \ref{rem_FC}.ii). Together with $\psi|_{B_{R}^{n-k}}=1$ this implies that { for $\epsilon$ small enough} there is a neighbourhood $B_\epsilon(x)\subset U_R(N)$ such that  $\sum_{\alpha\in I_N} h_\alpha(y)=1$ for all $y\in B_\epsilon(x)$.  Thus, $h_\beta|_{B_\epsilon(x)}=0$ and $h_\beta$ is smooth in $x$ for all $\beta\in I\setminus I_N$.

Moreover, by construction $\sum_{\beta\in I\setminus I_N} h_\beta + \sum_{\alpha\in I_N} h_\alpha= 1$. Hence, $(h_\gamma)_{\gamma\in I}$  gives a partition of unity subordinated to the Fermi coordinates.
The uniform boundedness of all $\uD^\a (h_{\beta}\circ \kappa_{\beta})$ follows from the uniform boundedness of all $\uD^\a (h_\alpha \circ \kappa_\alpha)$, $\uD^\a (h_{p_{\beta'}}^{\rm geo}\circ \kappa_{\beta'})$, $\uD^\a( \kappa_\alpha^{-1}\circ \kappa_\beta)$ and $\uD^\a ((\kappa_{\beta'})^{-1}\circ \kappa_\beta)$ together with  Remark \ref{rem_comp_T}.ii. 
\end{proof}

Collecting the last two lemmata we obtain immediately:

\begin{thm}\label{FC_admis} Let $(M,N)$ be of bounded geometry. Let $\uT^{\rm FC}$ be a trivialization of $M$ given by Fermi coordinates as in Definition \ref{FC} together with the subordinated partition of unity of Lemma \ref{part_unit_FC}. Then, $\uT^{\rm FC}$ is an admissible trivialization.
\end{thm}

\subsection{Trace Theorem}\label{trace_th_sec}

Let $(M^n,g)$ be a Riemannian manifold together with  an embedded submanifold $N^k$ where $k<n$. 
For $f\in \mathcal{D}(M)$ the trace operator is defined by pointwise restriction,  
$$
\Tr_N\, f:=f\big|_N.
$$
Let $X(M)$ and $Y(N)$ be some function or distribution spaces on $M$ and $N$, respectively. 
If $\Tr_N$ extends to a continuous map from $X(M)$ into $Y(N)$,  we say that the trace exists in $Y(N)$. If this extension is onto, we write $\Tr_N\, X(M)=Y(N)$. \\

For fractional  Sobolev spaces on manifolds we have the following trace result. 

\begin{thm}\label{trace-th}
Let $(M^n,g)$ be a  Riemannian manifold  together with an embedded $k$-dimensional submanifold $N$. Let  $(M,N)$ be of bounded geometry.  If $1<p<\infty$ and $s>\frac{n-k}{p}$, then $\Tr_{N}$ is a linear and bounded operator from $H^s_p(M)$ onto $B^{s-\frac{n-k}{p}}_{p,p}(N)$, i.e., 
\beq\label{trace-mfd}
\Tr_N \ H^s_p(M)=B^{s-\frac{n-k}{p}}_{p,p}(N). 
\eeq
\end{thm}

\begin{rem}\label{trace_rn}
For $(M,N)=(\mathbb{R}^n, \mathbb{R}^k)$ this is a classical result, cf. \cite[p.~138, Remark~1]{T-F1} and the references given therein. Here we think of $\mathbb{R}^k\cong  \{0\}^{n-k}\times \mathbb{R}^k \subset \mathbb{R}^n$. Furthermore, in \cite[p.~138, Remark~1]{T-F1} it is also shown that $\Tr_{\mR^k}$ has a { linear} bounded right inverse -- an extension operator $\text{Ex}_{\mR^n}$. \\ 
Note that $\Tr_{\mR^k}$ respects products with test functions, i.e., for $f\in H^s_p(\mR^n)$ and $\eta\in \mathcal{D}(\mR^n)$ we have $\Tr_{\mR^k}\, (\eta f)= \eta|_{\mR^k} \Tr_{\mR^k}\, f$. Moreover, if $\kappa$ is a diffeomorphism on $\R^n$ such that $\kappa(\mR^k)=\mR^k$, then $\Tr_{\mR^k}\, (f\circ \kappa)= \Tr_{\mR^k}\, f \circ \kappa|_{\mR^k}$. 
\end{rem}

\begin{proof}[Proof of Theorem \ref{trace-th}]
Via localization and pull-back we will reduce \eqref{trace-mfd} to the classical problem  of traces  on hyperplanes $\real^k$ in $\rn$. The proof is similar to \cite[Theorem~1]{Skr90}, but the Fermi-coordinates enable us to drop some of the restricting assumptions made there.  

By Theorem \ref{FC_admis} we have an admissible trivialization $\uT=(U_\alpha, \kappa_\alpha, h_\alpha)_{\alpha\in I}$ of $M$ by Fermi coordinates and the subordinated partition of unity from Lemma \ref{part_unit_FC}. Moreover, by the construction of the Fermi coordinates it is clear that $\kappa_\alpha^{-1}(N\cap U_{p^N_\alpha}) =\{0\}^{n-k}\times B_{2R}^k$ and, thus, their restriction to $N$  gives a geodesic trivialization  $\uT^{N, \rm geo}=(U'_\alpha:=U_\alpha\cap N, \kappa'_\alpha:=\kappa_\alpha|_{\kappa_\alpha^{-1}( U_\alpha')}, h'_\alpha:=h_\alpha|_{U'_\alpha})_{\alpha \in I_N}$  of $N$.

{\em 1. Step:} Let $f\in H^s_p(M)$. We define the trace operator via
\[ 
 (\Tr_N f) (x):=\sum_{\alpha\in I_N} \Tr_{\mR^k}\left[(h_\alpha f)\circ \kappa_\alpha\right] \circ (\kappa_\alpha')^{-1}(x), \qquad x\in N. 
\]
Note that $\Tr_N$ is well-defined since $(h_\alpha f)\circ \kappa_\alpha\in H_p^s(\mR^n)$ and  $\supp\, \Tr_{\mR^k}((h_\alpha f)\circ \kappa_\alpha)\subset V_\alpha'=V_{\alpha}\cap \mathbb{R}^k$. Moreover, for fixed $x\in N$ the summation is meant to run only over those $\alpha$ for which  $x\in U'_{\alpha}$.  Hence, the summation only runs over finitely many $\alpha$ due to the uniform locally finite cover. Obviously, $\Tr_N$ is linear and $\Tr_N|_{\mathcal{D}(M)}$ is given by the pointwise restriction. In order to show that $\Tr_N: H_p^s(M) \to B^{s-\frac{n-k}{p}}_{p,p}(N)$ is bounded, we set $A(\alpha):=\{\beta\in I_N\ |\ U_\alpha\cap U_\beta\neq \varnothing\}$.  Since the cover is uniformly locally finite, the number of elements in $A(\alpha)$ is bounded independent of $\alpha$. Together with Lemma \ref{Sob_Rn} and Remark \ref{trace_rn} we obtain

\begin{align}
\Vert\Tr_N& f\Vert_{B^{s-\frac{n-k}{p}}_{p,p}(N)}^p =  \sum_{\beta\in I_N} \Vert (h'_\beta \Tr_N f) \circ \kappa'_\beta \Vert^p_{B^{s-\frac{n-k}{p}}_{p,p}(\real^k)}\nonumber\\ 
\lesssim& \sum_{\beta\in I_N;\, \alpha\in A(\beta)} \Vert (h'_\beta \circ \kappa'_\beta) \left(\Tr_{\mR^k}\left[(h_\alpha f)\circ \kappa_\alpha\right] \circ \left[(\kappa_\alpha')^{-1} \circ \kappa'_\beta\right] \right)\Vert^p_{B^{s-\frac{n-k}{p}}_{p,p}(\real^k)}\nonumber\\
= & \sum_{\beta\in I_N;\, \alpha\in A(\beta)} \Vert \Tr_{\mR^k}\left[(h_\alpha h_\beta f)\circ \kappa_\alpha\right] \circ \left[(\kappa_\alpha')^{-1} \circ \kappa'_\beta\right] \Vert^p_{B^{s-\frac{n-k}{p}}_{p,p}(\real^k)}\nonumber\\
\lesssim & \sum_{\beta\in I_N;\, \alpha\in A(\beta)} \Vert \Tr_{\mR^k}\left[(h_\alpha h_\beta f)\circ \kappa_\alpha\right]  \Vert^p_{B^{s-\frac{n-k}{p}}_{p,p}(\real^k)}\nonumber\\
\lesssim & \sum_{\alpha\in I_N; \beta\in A(\alpha)} \Vert (h_\alpha h_\beta f)\circ \kappa_\alpha  \Vert^p_{H^{s}_{p}(\real^n)} \lesssim  \sum_{\alpha\in I_N} \Vert (h_\alpha f)\circ \kappa_\alpha  \Vert^p_{H^{s}_{p}(\real^n)},\label{Tr_est}
\end{align}

and hence, $\Vert\Tr_N f\Vert_{B^{s-\frac{n-k}{p}}_{p,p}(N)}^p \lesssim \Vert f\Vert^p_{H^{s}_{p}(M)}$, { where the involved constants do not depend on $f$}.

{\em 2. Step:} We will show that $\Tr_N$ is onto by constructing a right inverse -- an extension operator $\mathrm{Ex}_M$. Firstly, let  $\psi_1 \in \mathcal{D}(\mR^{k})$, and  $\psi_2 \in \mathcal{D}(\mR^{n-k})$ such that $\supp\ \psi_1 \in B_{2R}^k$, $\supp\ \psi_2 \in B_{2R}^{n-k}$, $\psi_1\equiv 1$ on $B_{R}^k$ and $\psi_2\equiv 1$ on $B^{n-k}_{\frac 32 R}$.  Then, we put $\psi:=\psi_1\times \psi_2\in \mathcal{D}(\mR^n)$. 

Let $f'\in B^{s-\frac{n-k}{p}}_{p,p}(N)$. Then we define the extension operator  by
\begin{align*}
 (\mathrm{Ex}_M f')(x):=\left\{ \begin{matrix}\sum_{\alpha\in I_N} \left[\psi \mathrm{Ex}_{\mR^n}( (h'_\alpha f')\circ \kappa_\alpha'  \right)]\circ \kappa_\alpha^{-1}(x), & x\in U_{2R}(N)\\
0,& \text{otherwise.}
\end{matrix}
\right.
\end{align*}
Note that the use of $\psi$ is to ensure that $\psi \mathrm{Ex}_{\mR^n}( (h'_\alpha f')\circ \kappa_\alpha'  )$ is compactly supported in $V_\alpha= B_{2R}^{k}\times B_{2R}^{n-k}$ for all $\alpha\in I_N$. 
Hence, one sees immediately that $\mathrm{Ex}_M$ is well-defined and  calculates  $\Tr_N (\mathrm{Ex}_M f')=f'$.

Thus, $\Tr_N$ is onto. Moreover, in order to show that $\mathrm{Ex}_M: B^{s-\frac{n-k}{p}}_{p,p}(N)\to H^s_p(M)$ is bounded,  we use Lemma \ref{Sob_Rn} and Remark \ref{trace_rn} again, which give 
  
\begin{align*}
\Vert\mathrm{Ex}_M f'\Vert_{\Hsp(M)}^p=& \sum_{\alpha\in I}\Vert(h_{\alpha} \mathrm{Ex}_M f')\circ \kappa_{\alpha}\Vert_{H^s_p(\rn)}^p
\lesssim \sum_{\alpha\in I;\, \beta\in A(\alpha)}\left\Vert \left(h_{\alpha} 
\left( \left[\psi \mathrm{Ex}_{\mR^n}( (h'_\beta f')\circ \kappa_\beta'  )\right]\circ \kappa_\beta^{-1}\right)
\right)\circ \kappa_{\alpha}\right\Vert^p_{H^s_p(\rn)}\\
\lesssim & \sum_{\beta\in I_N;\, \alpha\in I;\, U_\alpha\cap U_\beta\neq \varnothing}\Vert (h_\alpha \circ \kappa_\beta) \psi \mathrm{Ex}_{\mR^n}( (h'_\beta f')\circ \kappa_\beta'  )\Vert^p_{H^s_p(\rn)} \lesssim \sum_{\beta\in I_N}\Vert \mathrm{Ex}_{\mR^n}( (h'_\beta f')\circ \kappa_\beta'  )\Vert^p_{H^s_p(\rn)}\\
\lesssim & \sum_{\beta\in I_N}\Vert  (h'_\beta f')\circ \kappa_\beta'  \Vert^p_{B^{s-\frac{n-k}{p}}_{p,p}(\mR^k)}= \Vert f'\Vert^p_{B^{s-\frac{n-k}{p}}_{p,p}(N)}.
\end{align*}
Note that the estimate in the second line uses  $(h_\alpha\circ \kappa_\beta)\psi\in \mathcal{D}(V_\beta)$. This finishes the proof. 
\end{proof}

{
According to the coincidence with the classical Sobolev spaces $W^k_p(M)$ in the case of bounded geometry, cf. \eqref{coinc}, we obtain the following trace result as a consequence. 
}

{
\begin{cor}
Let $(M^n,g)$ be a  Riemannian manifold  together with an embedded $k$-dimensional submanifold $N$. Let  $(M,N)$ be of bounded geometry.  If $m\in \nat$, $1<p<\infty$, and $\ m>\frac{n-k}{p}$, then $\Tr_{N}$ is a linear and bounded operator from $W^m_p(M)$ onto $B^{m-\frac{n-k}{p}}_{p,p}(N)$, i.e., 
\[
\Tr_N W^m_p(M)=B^{m-\frac{n-k}{p}}_{p,p}(N). 
\]
\end{cor}
}

\begin{minipage}{0.63\textwidth}
\begin{ex} Our results generalize  \cite{Skr90} where traces were restricted to  submanifolds $N$ which had to be totally geodesic. By using Fermi coordinates we can  drop this extremely  restrictive assumption and cover  more (sub-)manifolds. \\For example, consider the case where $M$ is a surface of revolution of a curve $\gamma$ and $N$  a circle  obtained by the revolution of a fixed point $p\in M$. This resulting  circle is a geodesic if and only if the rotated curve has  an extremal point at $p$. But there is always a collar around $N$, hence, this situation  is also  covered by our assumptions. \end{ex}
\end{minipage}\hfill \begin{minipage}{0.35\textwidth}
\begin{psfrags}
    \psfrag{x}{\small $x$}
	\psfrag{y}{\small $y$}
	\psfrag{z}{\small $z$}
    \psfrag{p}{\small $p$}	
     \psfrag{N}{$N$}	
    \psfrag{M}{$M$}
    \psfrag{g}{$\gamma$}
	{\includegraphics[width=4.5cm]{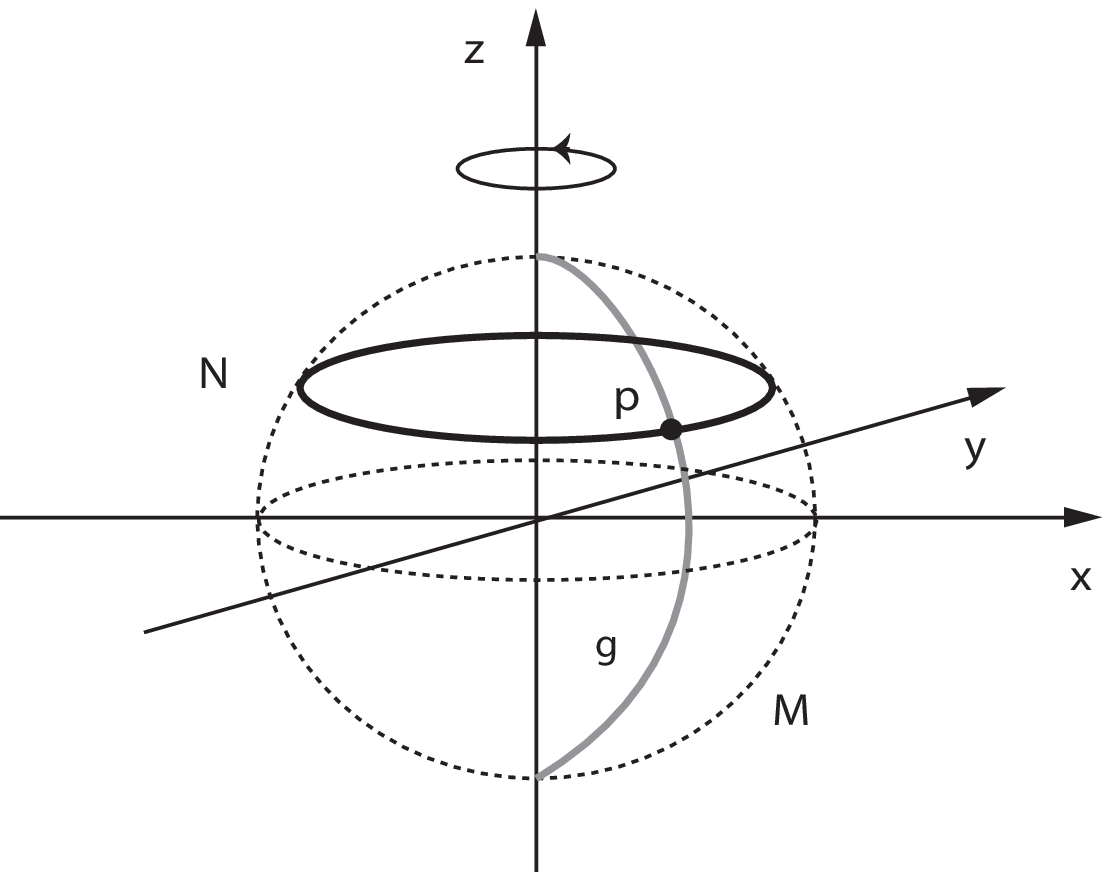}}
\end{psfrags}\\
\end{minipage}

\begin{rem}
We proved even more than stated. In Step 2 above it was shown that there exists a linear and bounded extension operator $\mathrm{Ex}_M$ from the trace space into the original space such that 
\[
\Tr_N\circ\mathrm{Ex}_M=\mathrm{Id},
\]
where $\mathrm{Id}$ stands for the identity in $B^{s-\frac{n-k}{p}}_{p,p}(N)$.  
 \end{rem} 

The first part of the  Trace Theorem \ref{trace-th} (i.e., the boundedness of the trace operator) can be extended to an even broader class of submanifolds. We give an example to illustrate the idea. 

\begin{ex} \label{ex-loccoll}
Let $(M^n,N_1^k)$ and $(M^n,N_2^k)$ be manifolds of bounded geometry with $N_1\cap N_2=\varnothing$. Set $N:=N_1\cup N_2$. Clearly, $\Tr_N=\Tr_{N_1}+\Tr_{N_2}$ (where  $\Tr_{N_1} f$ and $\Tr_{N_2} f$ are viewed as  functions on $N$ that equal zero on $N_2$ and $N_1$, respectively), and $\Tr_N$ is a linear bounded operator from $H_p^s(M)$ to $B_{p,p}^{s-\frac{n-k}{p}}(N)$. One may think of $N=\text{Graph}(x\mapsto x^{-1}) \cup x-\text{axis}\subset \mR^2$, where $(\mR^2, N)$ does not posses a uniform collar, cf. Definition \ref{bdd_geo}.iii. 

The boundedness of $\Tr_N$ is no longer expectable for an arbitrary infinite union of $N_i$, e.g., consider $N_i=\mR\times \{i^{-1}\}\subset \mR^2$, $i\in \mN$ and put  $f=\psi_1\times \psi_2$ with $\psi_1,\psi_2\in \mathcal{D}(\mR)$. Then $N=\sqcup_ iN_i\hookrightarrow \mR^2$ is an embedding, when $N$ is equipped with standard topology on each copy of $\mR$. But one cannot expect the trace operator to be bounded, since not every function $f\in C_c^\infty(\mR^2)$ restricts to a compactly supported function on $N$ ($N$ as a subset of $\mR^2$ is not intersection compact). \\
This problem can be circumvented when requiring that the embedded submanifold $N$ has to be a closed subset of $M$. However, even in this situation on can find  submanifolds $N$ for which the trace operator is not bounded in the sense of \eqref{Tr_est},  e.g. consider $N=\sqcup_{i\in \mN} \sqcup_{j=0}^{i-1} \mR\times \{i+\frac{j}{i}\}\hookrightarrow \mR^2$. 
\end{ex}

\begin{rem}\label{rem_th_loccoll} The above considerations give rise to the following  generalization of  Step 1 of the Trace Theorem \ref{trace-th}.
Assume that $N$ is a $k$-dimensional embedded submanifold of $(M^n,g)$ fulfilling (i), (ii) and (iv) of Definition \ref{bdd_geo} -- but not (iii). Lemmas \ref{FC-0-bound} and \ref{N_bdd_geo} remain valid, since their proofs do not use (iii). We replace (iii) with the following weaker version:
\begin{itemize}
 \item[(iii)$^\prime$] Let $(U'_\alpha=B_{2R}(p^N_\alpha))_{\alpha\in I_N}$ be a uniformly locally finite cover of $N$. Set $U_\alpha=F(B_{2R}^{n-k}\times B_{2R}^N(p_\alpha^N))$ as before. Then, $(U_\alpha)_{\alpha\in I_N}$ is a uniformly locally finite cover of $\cup_{\alpha \in I_N} U_\alpha$. 
\end{itemize}
Condition (iii)$^{\prime}$ excludes the negative examples from above . Furthermore, (iii)$^{\prime}$  together with the completeness of $N$  implies that $N$ is a closed subset of $M$.

With this modification, one can still consider  Fermi coordinates as in Definition \ref{FC}  but in general $U_{p_\alpha^N}\cap N\neq B_{2R}^N(p_\alpha^N)=: U'_{p_\alpha^N}$. Also the partition of unity can be constructed as in Lemma \ref{part_unit_FC} when making the following step in between: Following the proof of Lemma \ref{part_unit_FC} we define the map $\tilde{h}_\alpha=(\psi\times (h_\alpha' \circ \kappa_\alpha')) \circ \kappa_\alpha^{-1}$ for $\alpha\in I_N$. Since in general $\sum_{\alpha\in I_N} \tilde{h}_\alpha (x)$ can be bigger than one, those maps cannot be part of the desired partition of unity. Hence, we put $h_\alpha =\tilde{h}_\alpha(\sum_{\alpha'\in I_N} \tilde{h}_{\alpha'})^{-1}$ where $\sum_{\alpha'\in I_N} \tilde{h}_{\alpha'}\neq 0$ and $h_\alpha=0$ else. Smoothness and uniform boundedness of the derivatives of $h_\alpha$ follow as in  Lemma \ref{part_unit_FC}. Then one proceeds as before, defining $h_\beta$ for $I\setminus I_N$.\\
Now the proof of Step 1 of the Trace Theorem \ref{trace-th} carries over when replacing $h_\alpha$ by $\tilde{h}_\alpha$ in the definition of $\Tr_N$ and in the estimate  \eqref{Tr_est}. This leads to $\Vert \Tr_N f\Vert^p_{B_{p,p}^{s-\frac{n-k}{p}}(N)} \lesssim \sum_{\alpha\in I_N} \Vert(\tilde{h}_\alpha f)\circ \kappa_\alpha\Vert_{H_p^s(\mathbb{R}^n)}^p$. Finally, 

\begin{align*}\Vert \Tr_N f\Vert^p_{B_{p,p}^{s-\frac{n-k}{p}}(N)} \lesssim& \sum_{\alpha\in I_N} \Vert(\tilde{h}_\alpha f)\circ \kappa_\alpha\Vert_{H_p^s(\mathbb{R}^n)}^p = \sum_{\alpha\in I_N} \Vert(h_\alpha (\sum_{\beta\in I_N} \tilde{h}_{\beta})  f)\circ \kappa_\alpha\Vert_{H_p^s(\mathbb{R}^n)}^p\\
 \lesssim &\sum_{\alpha\in I_N;\ \beta\in A(\alpha)} \Vert(h_\alpha \tilde{h}_{\beta}  f)\circ \kappa_\alpha\Vert_{H_p^s(\mathbb{R}^n)}^p \lesssim\sum_{\alpha\in I_N} \Vert(h_\alpha f)\circ \kappa_\alpha\Vert_{H_p^s(\mathbb{R}^n)}^p\leq \Vert f\Vert_{H_p^s(M)}^p,
\end{align*}
demonstrates the boundedness of the trace operator $\Tr_N$ under this generalized assumptions on the submanifold $N$.
\end{rem}

\section{Vector bundles}\label{sec_vec_bu}

The results about function spaces on manifolds of bounded geometry obtained so far can be transferred to certain vector bundles. For that we need  a  concept of bounded geometry for vector bundles. After giving such a definition, we shall proceed along the lines of the previous section --  introducing synchronous trivialization along geodesic normal coordinates and  Fermi coordinates and stating a corresponding Trace Theorem.

\subsection{Vector bundles of bounded geometry}

\begin{defi}\label{def_vec_bddgeo}\cite[Section A1.1]{Shu} Let $E$ be a vector bundle over a Riemannian manifold $(M,g)$ of bounded geometry together with an atlas $\uA_E=(U^{\rm geo}_\alpha, \kappa^{\rm geo}_\alpha, \xi_\alpha)_{\alpha \in I}$ where $(U^{\rm geo}_\alpha, \kappa^{\rm geo}_\alpha)_{\alpha\in I}$ is a geodesic atlas of $M$ as in Example \ref{geod_coord_triv}. Let $\tilde{\mu}_{\alpha\beta}$ denote the transition functions belonging to $\xi_\alpha, \xi_\beta$. The vector bundle $E$ together with the choice of an atlas $\uA_E$ is said to be of bounded geometry if for all $k\in \mN_0$ there is a constant $C_k$ such that $|\uD^\a \tilde{\mu}_{\alpha\beta}|\leq C_k$ for all $\alpha, \beta\in I$ with $U^{\rm geo}_\alpha\cap U^{\rm geo}_\beta\neq \varnothing$ and all multi-indices $\a$ with $|\a|\leq k$.
\end{defi}

We give an example of a special trivialization $\xi_\alpha$:

\begin{defi}[\bf Synchronous trivialization along geodesic normal coordinates]\label{def_syn_geo}\hfill\\
Let $(E,\nabla^E,\<.,.\>_E)$ be a Riemannian or hermitian vector bundle over a Riemannian manifold $(M^n,g)$. Let $M$ be of bounded geometry, and let the connection $\nabla^E$ be metric.
Let $\uA^{\rm geo}=(U^{\rm geo}_\alpha, \kappa^{\rm geo}_\alpha)_{\alpha \in I}$ be a geodesic atlas of $M$ as in Example \ref{geod_coord_triv}, and let $p_\alpha$ denote the center of the ball $U^{\rm geo}_\alpha$. The choice of the orthonormal frame on $T_{p_\alpha}M$ -- {already used in the definition of the geodesic coordinates, cf. Example \ref{geod_coord_triv} --  is again denoted by ${\lambda}_\alpha: \mR^n\to T_{p_\alpha}M$.}  We choose an orthonormal frame $(\tilde{e}_1(p_\alpha), \ldots, \tilde{e}_r(p_\alpha))$ for each $E_{p_\alpha}$ ($\alpha\in I$).  Then, $E|_{U^{\rm geo}_\alpha}$ is trivialized by parallel transport along radial geodesics emanating from $p_\alpha$ as follows: For $1\leq \rho\leq r$, let $X_\rho^ v(t)\in E_{c_v(t)}$  be the unique solution of the differential equation  $\nabla^E_{\dot{c}_v}X_\rho^v=0$ with $X_\rho^ v(0)=\tilde{e}_\rho(p_\alpha)$ and $c_v(t)$ being the unique geodesic with $c_v(0)=p_\alpha$ and $\dot{c}_v(0)=v\in T_{p_\alpha}^{\leq r} M$, where $r$ is smaller 
than the injectivity radius of $M$.
Then the trivialization by parallel transport is given by
$\xi_\alpha^{\rm geo}: (x,u)\in V_\alpha^{\rm geo}\times \mC^r \mapsto u^\rho X_\rho^{\lambda_\alpha(x)}(1)\in E|_{U_\alpha^{\rm geo}}$ and is called synchronous trivialization (along geodesic normal coordinates). $\uA_E^{\rm geo}=(U^{\rm geo}_\alpha, \kappa^{\rm geo}_\alpha, \xi_\alpha^{\rm geo})_{\alpha \in I}$ is called a geodesic atlas of $E$.
\end{defi}

Note that by construction, $h_{\sigma\tau}(0)=\delta_{\sigma\tau}$ for all $\alpha\in I$. Since $\tilde{e}_\sigma$ on $U_\alpha$ is obtained by the parallel transport for a metric connection, we get  $h_{\sigma\tau}=\delta_{\sigma\tau}$ on each $U_\alpha$ and, hence, $\tilde{\Gamma}_{i\sigma}^\rho=-\tilde{\Gamma}_{i\rho}^\sigma$, cp. \eqref{h_ode}. 

\begin{rem}\label{rem_vec_bddgeo_Eich}
In \cite[Section 1.A.1]{Eichb} one can find another definition of $E$ being of bounded geometry: A hermitian or Riemannian vector bundle $(E, \nabla^E, \<.,.\>_E)$ over $(M,g)$ with metric connection is of bounded geometry, if $(M,g)$ is of bounded geometry and if the curvature tensor of $E$ and all its covariant derivatives are uniformly bounded. 
In \cite[Theorem B]{Eich} it was shown that bounded geometry of $E$ in the sense of \cite[Section 1.A.1]{Eichb} is equivalent to
the following condition:

For all $k\in \mN_0$ there is a constant $C_k$ such that for all $\alpha\in I$, $1\leq i\leq n$, $1\leq \rho,\sigma\leq r$ and all multi-indices $\a$ with $|\a|\leq k$,
\begin{equation} |\uD^\a \tilde{\Gamma}_{i\rho}^\sigma|\leq C_k,\label{coord_E_bdd_Gamma}\end{equation}  where $\tilde{\Gamma}_{i\rho}^\sigma$ denote the Christoffel symbols with respect to $\xi^{\rm geo}_\alpha$.
\end{rem}

Our next aim is to compare the two definitions of bounded geometry of $E$ given above: 

\begin{thm}
Let $(E, \nabla^E, \<.,.\>_E)$ be a hermitian or Riemannian vector bundle over a Riemannian manifold $(M,g)$. Let $(M,g)$ be of bounded geometry, and let $\nabla^E$ be a metric connection. Moreover, let  $\uA_E^{\rm geo}=(U^{\rm geo}_\alpha, \kappa^{\rm geo}_\alpha, \xi^{\rm geo}_\alpha)_{\alpha\in I}$ be a geodesic atlas of $E$, see Definition \ref{def_syn_geo}. Then, $E$ together with $\uA_E^{\rm geo}$ is of bounded geometry in the sense of Definition \ref{def_vec_bddgeo} if, and only if, it is of bounded geometry in the sense of \cite[Section 1.A.1]{Eich}, cf. Remark \ref{rem_vec_bddgeo_Eich}.
\end{thm}

\begin{proof}
 
 Let $\alpha, \beta \in I$. For simplicity, we assume that $p_\alpha, p_\beta \in U^{\rm geo}_\alpha\cap U^{\rm geo}_\beta$. Then, for all $z\in   U^{\rm geo}_\alpha\cap U^{\rm geo}_\beta$ the geodesics joining $z$  with $p_\alpha$ and $p_\beta$, respectively, are completely contained in $ U^{\rm geo}_\alpha\cap U^{\rm geo}_\beta$. For atlases not satisfying this assumption, one can switch to a refined atlas { and use the composition of pairs of charts, each pair satisfying the assumption from above.}
 
 Let $v\in T_{p_\alpha}M$ and let $X_\rho^ v(t)$ and $c_v(t)$ be defined as above.  Let $\tilde{\Gamma}_{i\rho}^\sigma$ be the Christoffel symbols for $\nabla^E$ with respect to $\xi_\beta^{\rm geo}$. We put $Y^v(t)=(\xi^{\rm geo}_\beta)^{-1}X^{v}(t)$. Moreover, let {$\Phi_1(t):=\Phi_1(t,0,(\kappa_\beta^{\rm geo})^{*}(v))=(\kappa_\beta^{\rm geo})^{-1} c_v(t)$} be the geodesic flow on $(U^{\rm geo}_\beta, \kappa^{\rm geo}_\beta)$. Then the initial value problem $\nabla^E_{\dot{c}_v} X^v_\rho=0$ with $X^v_\rho(0)=\tilde{e}_\rho(p_\alpha)$ reads in local coordinates as $\partial_t Y^v_\sigma+ \dot{\Phi}_1^i\tilde{\Gamma}_{i\sigma}^\rho Y^v_\rho=0$ with $Y^v_\sigma(0)= (\xi^{\rm geo}_\beta)^{-1} \tilde{e}_\sigma(p_\alpha)$.  We denote the corresponding flow by $\Phi^v(t, Y^v_\rho(0))$ and have {
 $\xi_\alpha^{\rm geo}(x,u)= \xi_\beta^{\rm geo} \left( \sum_{i=1}^r u^\rho \Phi^{\lambda_\alpha(x)} (1, {(\xi_\beta^{\rm geo})^{-1}}\tilde{e}_\rho(p_\alpha) \right)$.}  
 By \eqref{coord_E_bdd_Gamma}, $\tilde{\Gamma}_{i\rho}^\sigma$ and all its derivatives are uniformly bounded. Moreover, the same is true for the geodesic flow $\dot{\Phi}_1$ since $(M,g)$ is of bounded geometry. Then by Lemma \ref{flow_lem}, $E$ together with $\uT^{\rm geo}$ is bounded in the sense of Definition \ref{def_vec_bddgeo}.

Conversely, let $E$ be a vector bundle of bounded geometry in the sense of Definition \ref{def_vec_bddgeo}.  Since $\xi_\alpha^{\text{geo}}$ is a synchronous trivialization, $h_{\rho\sigma}=\delta_{\rho\sigma}$, see Definition \ref{def_syn_geo} and below. Let now $p$ be any point in $M$ and $\kappa$  geodesic coordinates on a ball around $p$ with radius $r$. Let $V$ be a unit radial vector field starting at $p$. Then its derivatives are uniformly bounded at distances between $\frac{r}{10}$ and $r$ from $p$, since $(M,g)$ is of bounded geometry. For a point $q\in M$, let $v_i$ be $n$ unit vectors that span $T_qM$. We set $p_i=\exp^M_q (\frac{r}{2}v_i)$. Let $(\tilde{e}_\sigma^i(p_i))_\sigma$ be an orthonormal frame of $E_{p_i}$. We consider geodesic normal coordinates and a synchronous trivialization around those $p_i$. 
{Moreover, let $\hat{v}_i$ be the vector $v_i$ parallel transported to $p_i$ along $c_{v_i}$.} Since the transition functions of $E$ and all its derivatives are uniformly 
bounded, $\tilde{e}_\sigma^ i(\exp^M_{p_i}(-t\hat{v}_i)$ and its derivatives are uniformly bounded for $t\in (\frac{r}{10},r)$. 
In particular,  uniformly bounded means in this context, that the bound may depend on the order of the derivatives but not on $i$. Moreover, since the synchronous trivialization is defined by parallel transport along radial geodesics, {we have $\nabla^E_{v_i} \tilde{e}^i_\sigma(q)=0$ for all $i$ and $\sigma$.}
For a synchronous trivialization over $q$, those equations give a linear system on the Christoffel symbols $\tilde{\Gamma}_{l\tau}^\rho$, whose coefficients are polynomials in the components of  $\tilde{e}^i_\sigma(q)$, their first derivatives and $v_i$ with respect to the geodesic coordinates around $q$. But those are uniformly bounded as explained above. Moreover, by construction, this system has a unique solution. Hence, the Christoffel symbols and all its derivatives in the synchronous trivialization around $q$ are uniformly bounded.  
\end{proof}

\begin{ex}
\begin{itemize}
 \item[(i)] Let $(M,g)$ be a Riemannian manifold of bounded geometry. Then its tangent bundle equipped with its Levi-Civita connection is trivially of bounded geometry.
 \item[(ii)] Let $(M,g)$ be a Riemannian spin manifold of bounded geometry with chosen spin structure, i.e. we have chosen a double cover $P_{\mathrm{spin}}M$ of the oriented orthonormal frame bundle such that it is compatible to the double covering $\mathrm{Spin}(n)\to \mathrm{SO}(n)$, cf. \cite[Section 1.5 and 2.5]{Frie}.  We denote by $S=P_{\mathrm{spin}}(M)\times_\kappa \mC^{\left[\frac{n}{2}\right]} $ the associated spinor bundle,  
where $\kappa: \mathrm{Spin}(n) \to U(\mC^{\left[\frac{n}{2}\right]})$ is the spin representation, cf. \cite[Section 2.1]{Frie}. The connection on $S$ is induced by the Levi-Civita connection on $M$. Hence, the 
Riemannian curvature of $S$ and all its covariant derivatives are uniformly bounded. 
In this spirit, any natural vector bundle $E$ over a manifold of bounded geometry equipped with a geodesic trivialization of $E$ is of bounded geometry.
 \item[(iii)]  Let $(M,g)$ be a Riemannian $\mathrm{Spin}^{\mC}$ manifold of bounded geometry. Here the spinor bundle $S$ described above may not exist globally (but it always exists locally). But a $\mathrm{Spin}^{\mC}$-structure assures the existence of a $\mathrm{Spin}^{\mC}$-bundle $S'$ that is a hermitian vector bundle of rank $2^{\left[\frac{n}{2} \right]}$, endowed with a natural scalar product and with a connection $\nabla^{S'}$ that parallelizes the metric. Moreover, the $\mathrm{Spin}^{\mC}$-bundle is endowed with a Clifford
multiplication denoted by '$\cdot$', where $\cdot: TM \to \text{End}_{\mC}(S')$ is such that at every point $x\in M$
'$\cdot$' defines an irreducible representation of the corresponding Clifford algebra. The determinant line bundle $\text{det}\, S'$ has a root of index $2^{\left[\frac{n}{2}\right]-1}$ -- denoted by $L$ and called the auxiliary line bundle associated to the $\mathrm{Spin}^{\mC}$-structures, \cite[Section 2.5]{Frie}. The square root of $L$ always exists locally but $S'=S\otimes L^\frac{1}{2}$ is defined even globally, \cite[Appendix D]{Frie}. The connection on $S'$ is the twisted connection of the one on the spinor bundle coming from the Levi-Civita connection (as described in (ii)) and a connection on $L$. Hence, for $S'$ being of bounded geometry, we  not only need that $(M,g)$ is a bounded geometry but also that the curvature of the auxiliary line bundle and its covariant derivatives are uniformly bounded.
\end{itemize} 
\end{ex}

\subsection{Sobolev spaces on vector bundles}

We start with two definitions of Sobolev spaces on vector bundles $E$ over $M$. The first one is for vector bundles of bounded geometry only.

\begin{defi}\cite[Section A1.1]{Shu} Let $E$ with trivialization  $\uT_E=(U^{\rm geo}_\alpha, \kappa^{\rm geo}_\alpha, \xi_\alpha, h^{\rm geo}_\alpha)_{\alpha\in I}$  be a vector bundle of bounded geometry over a Riemannian manifold $(M^n, g)$. Then, for $s\in \real$, $1<p<\infty$, the Sobolev space $\mathcal{H}_p^s(M,E)$ contains all distributions $\phi\in \mathcal{D}'(M,E)$ with \[\Vert\phi\Vert_{\mathcal{H}_p^s(M,E)}:= \left( \sum_{\alpha\in I} \Vert \xi_\alpha^* (h^{\rm geo}_\alpha \phi)\Vert_{H_p^s(\mR^n, \mC^r)}^p \right)^\frac{1}{p}<\infty,\]
where $r$ is the rank of $E$.
\end{defi}

Let $E$ be a hermitian or Riemannian vector bundle over a complete Riemannian manifold $(M,g)$ of rank $r$ with fiber product $\<.,.\>_E$ and  connection $\nabla^E: \Gamma(TM)\otimes \Gamma(E)\to \Gamma(E)$. \\
In general, $\xi_\alpha$ and $\nabla^E$  have nothing to do with each other. But one can alternatively use the connection in order to define Sobolev spaces: 
For $k\in \mN_0$, $1<p<\infty$, let the $W_p^k(M,E)$-norm be defined by
\[ \Vert \phi\Vert_{{W}_p^k(M,E)}^p=\sum_{i=1}^k \int_M | \underbrace{\nabla^E\cdots \nabla^E}_{i\ \mathrm{times}} \phi |^p \vo_g\ \quad \mathrm{for}\ \phi\in \mathcal{D}(M,E).\]
Then the space ${W}_p^k(E)$ is defined to be the completion of $\mathcal{D}(M,E)$ with respect to the ${H}_p^k(M,E)$-norm.

\begin{thm}Let $(E, \nabla^E, \<.,.\>_E)$ be of bounded geometry. In case that $\xi_\alpha$ is the synchronous trivialization along geodesic normal coordinates $W_p^k(M,E)=\mathcal{H}_p^k(M,E)$ for all $k\in \mN_0$ and $1<p<\infty$. 
\end{thm}

\begin{proof} We briefly sketch the proof which is straightforward.  Let $\phi\in \mathcal{D}(U_\alpha, E|_{U_\alpha})$. By induction we have  
\[ ((\nabla^E)^k\phi)_{i_1,\ldots,i_k}:= \xi_\alpha^* \left( \nabla^E_{e_{i_k}}\cdots \nabla^E_{e_{i_1}} \phi\right)  = \sum_{l\leq k} d_{j_1,\ldots, j_l} \partial_{j_1}\cdots \partial_{j_l} (\xi_\alpha^*(\phi)),\] where the coefficients $d_{j_1,\ldots, j_l}$ are itself polynomials in $g_{ij}$, $g^{ij}$, $\tilde{\Gamma}_{i\sigma}^\rho$ and their derivatives (and depend on $i_1,\ldots, i_k$).  Moreover, again by induction, one has that the coefficients of the leading terms, i.e., $l=k$, are given by 
$d_{j_1,\ldots, j_k}=g^{i_1j_1}\cdots g^{i_kj_k}$.

By Remark \ref{rem_bdgeom}.iii, all those coefficients are uniformly bounded. Moreover, using the fact that $\xi_\alpha$ is obtained by synchronous trivialization, we have $h_{\rho\sigma}=\delta_{\rho\sigma}$, see below Definition \ref{def_syn_geo}. Hence, there are constants $\tilde{C_k}>0$ with 
\begin{align*} |(\nabla^E)^k \phi|_E^2=& \sum_\sigma g^{i_1j_1} \cdots g^{i_kj_k}((\nabla^E)^k \phi)^\sigma_{i_1, \ldots, i_k}((\nabla^E)^k \phi)^\sigma_{j_1, \ldots, j_k}\circ \kappa_\alpha^{-1}\\
\leq& \tilde{C}_k \sum_{\gamma_l; |\gamma_l|\leq k; l\in{1,2}} |D^{\gamma_1} (\xi_\alpha^*(\phi))||D^{\gamma_2} (\xi_\alpha^*(\phi))| 
\end{align*}

  for all $\alpha$ and all $\phi\in \mathcal{D}(U_\alpha, E|_{U_\alpha})$. 
Together with a uniform upper bound on $\det g_{ij}$ which  follows again from Remark \ref{rem_bdgeom}.iii,  we obtain 
$\Vert (\nabla^E)^k \phi (e_{i_1}, \ldots, e_{i_k})\Vert_{L_p(U_\alpha, E|_{U_\alpha})}\leq \tilde{C} \sum_{\gamma, |\gamma|\leq k} \Vert D^\gamma (\xi_\alpha^*(\kappa_\alpha))\Vert_{L_p(V_\alpha, \mF^r)}$.\\
On the other hand, by the remark on the leading coefficients $d$ from above 
\[g_{i_1m_1}\cdots g_{i_km_k} \xi_\alpha^*((\nabla^E)^k \phi (e_{i_1}, \ldots, e_{i_k}))= \partial_{m_1}\cdots \partial_{m_k} (\xi_\alpha^*(\phi)) + \mathrm{terms\ with\ lower\ order\ derivatives}.\]
Thus, as above
\[  \partial_{m_1}\cdots \partial_{m_k}(\xi_\alpha^*(\phi))=\sum_{l\leq k} d'_{i_1, \ldots, i_l} \xi_\alpha^*( (\nabla^E)^l \phi(e_{i_1}, \ldots, e_{i_l})).\]

where the functions $d'_{i_1, \ldots, i_l}: V_\alpha \to \mathbb{R}$ are again polynomials in $g_{ij}$, $g^{ij}$, $\tilde{\Gamma}_{i\rho}^\sigma$ and their derivatives.

In the same way as before we obtain
\[\Vert \xi_\alpha^*(\phi)\Vert_{H_p^s(V_\alpha, \mF^r)}\leq C \sum_{l\leq k} \Vert (\nabla^E)^l \phi\Vert_{L_p(U_\alpha, E|_{U_\alpha})}.\]

 Let now $\phi\in \mathcal{D}(M, E)$. Then, using Example \ref{geod_triv}, the uniform local finiteness of the cover and the local inequalities from above we see that for  $k\in \mathbb{N}_0$,

\begin{align*}
\big\Vert (\nabla^E)^k \phi\big\Vert_{L_p(M,E)}^p =&  \sum_{\alpha\in I} \big\Vert \xi_\alpha^*(  h_\alpha^{\rm geo} (\nabla^E)^k \phi) \big\Vert_{L_p(V_\alpha,\mF^r)}^p \\
=&\sum_{\alpha\in I} \Big\Vert \xi_\alpha^*\left(  (\nabla^E)^k (h_\alpha^{\rm geo} \phi) - \sum_{i=1}^k {k\choose i}(\nabla^M)^ih_\alpha^{\rm geo} (\nabla^E)^{k-i} \phi \right)\Big\Vert_{L_p(V_\alpha,\mF^r)}^p\\
\lesssim& \sum_{\alpha\in I} \left( \Vert  (\nabla^E)^k (h_\alpha^{\rm geo} \phi)\Vert_{L_p(U_\alpha,E|_{U_\alpha})}^p + \sum_{i=1}^k \Vert (\nabla^M)^ih_\alpha^{\rm geo} (\nabla^E)^{k-i} \phi\Vert_{L_p(U_\alpha,E|_{U_\alpha})}^p\right)\\
\lesssim& \sum_{\alpha\in I}  \Vert  (\nabla^E)^k (h_\alpha^{\rm geo} \phi)\Vert_{L_p(U_\alpha,E|_{U_\alpha})}^p + \sum_{i=1}^k \Vert (\nabla^E)^{k-i} \phi \Vert_{L_p(M,E)}^p \\
\lesssim& \sum_{\alpha\in I}  \Vert  \xi_\alpha^*(h_\alpha^{\rm geo} \phi)\Vert_{H_p^k(\mathbb{R}^n, \mF^r)}^p + \sum_{i=1}^k  \Vert (\nabla^E)^{k-i} \phi\Vert_{L_p(M,E)}^p. 
\end{align*}

Using this estimate inductively, there is a constant $C'>0$ with 
\[\Vert \phi\Vert_{W_p^k(M,E)}\leq C' \Big(\sum_{\alpha\in I} \Vert  \xi_\alpha^*(h_\alpha^{\rm geo} \phi)\Vert^p_{H_p^k(\mathbb{R}^n, \mF^r)}\Big)^{1/p}. \]

On the other hand, 

\begin{align*}
\sum_{\alpha\in I} \Vert \xi_\alpha^*(h_\alpha^{\rm geo} \phi)\Vert^p_{H_p^k(\mathbb{R}^n,\mF^r)}&\lesssim  \sum_{\alpha\in I} \sum_{i=0}^k \Vert  (\nabla^E)^i(h^{\rm geo}_\alpha \phi)\Vert_{L_p(U_\alpha, E|_{U_\alpha})}^p
 \lesssim \sum_{\alpha\in I} \sum_{i=0}^k\sum_{j=0}^i \Vert  (\nabla^M)^j h_\alpha^{\rm geo} (\nabla^E)^{i-j}\phi\Vert_{L_p(U_\alpha, E|_{U_\alpha})}^p\\
& \lesssim  \sum_{\alpha\in I} \sum_{i=0}^k\sum_{j=0}^i \Vert  (\nabla^E)^{i-j} \phi\Vert_{L_p(U_\alpha, E|_{U_\alpha})}^p
 \lesssim  \sum_{\alpha\in I}  \Vert  \phi\Vert_{W_p^k(U_\alpha, E|_{U_\alpha})}^p\lesssim  \Vert  \phi\Vert_{W_p^k(M, E)}^p.
\end{align*}
The coincidence of the corresponding spaces follows since $\mathcal{D}(M, E)$  is dense in  $W^k_p(M,E)$  for $k\in \mN$, cf. \cite[Theorem~4.3]{strich}.
\end{proof}

The above considerations give rise to the following definition.

\begin{defi} \label{def_vec_norm_bddgeo}
Let $(E, \nabla^E, \<.,.\>_E)$ be of bounded geometry. In case that $\xi_\alpha$ is the synchronous trivialization along geodesic normal coordinates we set {$H_p^s(E):=W_p^s(M,E):=\mathcal{H}_p^s(M,E)$} for all $s\in \real$ and $1<p<\infty$.

\end{defi}

\subsection{Sobolev norms on vector bundles of bounded geometry via trivializations}

As for Sobolev spaces on manifolds we look for 'admissible' trivializations of a vector bundle $E$ such that the resulting Sobolev norms are equivalent to those  obtained when using a geodesic trivialization of $E$.

\begin{defi}\label{H-koord-vect}
 Let $E$ be a hermitian or Riemannian vector bundle  of rank $r$ over $(M^n,g)$ with a uniformly locally finite trivialization $\uT_E=(U_\alpha, \kappa_{\alpha},\xi_\alpha, h_{\alpha})_{\alpha\in I} $. Furthermore, let $s\in \real$ and $1< p<\infty$.  Then the space $H^{s,\uT_E}_p(E)$ contains all distributions $\phi\in \mathcal{D}'(M,E)$ such that 
\begin{align*}
\Vert \phi\Vert_{H^{s,\uT_E}_{p}(E)}:=\left(\sum_{\alpha\in I} \Vert (\xi_\alpha)^* (h_\alpha \phi)\Vert^p_{H^s_p(\mathbb{R}^n, \mF^r)}\right)^{\frac{1}{p}}
\end{align*}
is finite. 
\end{defi}

\begin{defi}\label{bddcoord2} Let $(E, \nabla^E, \<.,.\>_E)$ be a hermitian or Riemannian vector bundle of rank $r$ and of bounded geometry over a Riemannian manifold $(M^n,g)$. Let  $\uT_E=(U_\alpha,  \kappa_\alpha, \xi_\alpha, h_\alpha)_{\alpha \in I}$ be a uniformly locally finite trivialization of $E$. Using the notations from above, we say that $\uT_E$ is an admissible trivialization for $E$ if the following are fulfilled:
\begin{itemize}
 \item[(C1)]$\uT:=(U_{\alpha}, \kappa_\alpha, h_\alpha)_{\alpha\in I}$ is an admissible trivialization of $M$.
 \item[(C2)]$\uT_E$ is compatible with the synchronous trivialization along geodesic coordinates, i.e., for $\uA_E^{\rm geo}=(U^{\rm geo}_\beta, \kappa^{\rm geo}_\beta, \xi_\beta^{\rm geo})_{\beta \in J}$ being a geodesic atlas of $E$, cf. Definition \ref{def_syn_geo},
  there are constants $C_k>0$ for $k\in \mN_0$ such that for all $\alpha\in I$ and $\beta\in J$ with $U_{\alpha}\cap U^{\rm geo}_{\beta}\neq \varnothing$ and all $\a\in \mN_0^n$ with $|\a|\leq k$,  
  \[ |\uD^\a \tilde{\mu}_{\alpha\beta}|\leq C_k \qquad \text{\ and\ } \qquad |\uD^\a \tilde{\mu}_{\beta\alpha}|\leq C_k.\]
\end{itemize}
\end{defi}

For vector bundles of bounded geometry we have corresponding  results as on manifolds of bounded geometry. We start with the formulation of  the analog of Theorem \ref{indep_H}. The proof follows in the same way.

\begin{thm}\label{indep_H_vec} 
Let $E$ be a hermitian or Riemannian vector bundle over a Riemannian manifold $(M,g)$. Let $\uT_E=(U_{\alpha},\kappa_{\alpha},\xi_\alpha, h_\alpha)_{\alpha \in I}$ be an admissible trivialization of $E$. Furthermore, let $s\in \real$ and   $1<p<\infty$. Then, 
\beq\label{indep_H_1_vec}
H^{s,\uT_E}_{p}(E)=H^{s}_{p}(E). 
\eeq
\end{thm}

\subsection{Trace Theorem for vector bundles}

\begin{defi}[\bf Synchronous trivialization along Fermi coordinates]\label{FC_vec} 
Let $(M,N)$ be of bounded geometry, and let $E$ be a hermitian or Riemannian vector bundle of bounded geometry over $M$.
Let $\uT^{FC}=(U_\gamma, \kappa_\gamma, h_\gamma)_{\gamma\in I}$ be a trivialization of $M$ using Fermi coordinates (adapted to $N$). We refer to   Section \ref{FC_sec} (also concerning the notation). In case that $\gamma\in I\setminus I_N$, we trivialize $E|_{U_\gamma}$ via synchronous trivialization along the underlying geodesic coordinates as described in Definition \ref{def_syn_geo}.  In case that $\gamma\in I_N$, we first trivialize $E|_{U_\gamma\cap N}$ along the underlying geodesic coordinates on $N$. Then, we trivialize by parallel transport along geodesics emanating at $N$ and being normal to $N$. The resulting trivialization is denoted by $\uT^{FC}_E=(U_\gamma, \kappa_\gamma, \xi_\gamma, h_\gamma)_{\gamma\in I}$.
\end{defi}
 
Next, we state corresponding results to Lemma \ref{FC-1} and Theorem \ref{trace-th}. 

\begin{lem}\label{FC_1_vec}
 The trivialization $\uT_E^{\rm FC}$ introduced in Definition \ref{FC_vec} fulfills condition (C2).
\end{lem}

\begin{proof}
 The proof is the same as in  Lemma \ref{FC-1}.
\end{proof}

\begin{thm}\label{trace-th_vec}
Let $E$ be a hermitian or Riemannian vector bundle of bounded geometry over a Riemannian manifold $(M,g)$  together with an embedded $k$-dimensional submanifold $N$. Let $(M,N)$ be of bounded geometry.  If $1<p<\infty$ and $s>\frac{n-k}{p}$, then the pointwise restriction $\Tr_{N}: \mathcal{D}(M,E)\to \mathcal{D}(N,E|_N)$  extends to a linear and bounded operator from $H^s_p(E)$ onto $B^{s-\frac{n-k}{p}}_{p,p}(E|_N)$, i.e., 
\beq\label{trace-mfd_vec}
\Tr_N \, H^s_p(E)=B^{s-\frac{n-k}{p}}_{p,p}(E|_N). 
\eeq
Moreover, $\Tr_N$ has a linear and bounded right inverse, an extension operator $\mathrm{Ex}_M: B^{s-\frac{n-k}{p}}_{p,p}(E|_N) \to H^s_p(E)$.
\end{thm}

\begin{proof}
We start with the case that $E=\mR^n\times  \mF^r$ is the trivial bundle over $\mR^n$. In this case the claim follows immediately from the Trace Theorem on $(\mR^n, \mR^k)$ and Lemma \ref{vec_norm_equ}.

 The rest of the proof follows along the lines of Theorem \ref{trace-th}, using that by construction 
$(U'_\gamma=U_\gamma\cap N, \kappa'_\gamma=(\kappa_\gamma^{-1}|_{U'_\gamma})^{-1}, \xi'_\gamma=\xi_\gamma|_{(V_\gamma\cap \mR^k)\times \mF^r}, h'_\gamma= h_\gamma|_{U'_\gamma})_{\gamma\in I_N}$ gives a geodesic trivialization of $E|_N$.
\end{proof}

\section{Outlooks} 

\subsection{Spaces with symmetries - a first straightforward example}\label{outlook_1}

The aim of  this subsection is to give an  application of admissible trivializations to spaces with symmetries. We consider manifolds $M$, where a countable discrete group $G$ acts in a convenient way and show 
that  the Sobolev spaces of functions on the resulting orbit space $M/G$ and the weighted Sobolev spaces of $G$-invariant functions on $M$ coincide. This is in spirit of 

\begin{thm} \cite[Section 9.2.1]{T-F1}  Let $1<p<\infty$ and consider the weight  $\rho(x)=(1+|x|)^{-\varkappa}$ on Euclidean space $\mR^n$ where $\varkappa p>n$. Let $\mathbb{T}^n:=\mR^n/\mZ^n$ denote the torus and  $\pi: \mR^n\to \mathbb{T}^n$  the natural projection. Put $H_{p,\pi}^s(\mR^n, \rho):=\{ f\in \mathcal{D'}(\mR^n) \ |\  \rho f\in  H_{p}^s(\mR^n)\ \text{and\ } f \text{\ is\ $\pi$-periodic} \}$, then
\[ H_p^s(\mathbb{T}^n)=H_{p,\pi}^s(\mR^n, \rho).\]
 \end{thm}
 
This is just a special case of the theorem given in \cite[Section 9.2.1]{T-F1}, where more generally Besov and Triebel-Lizorkin spaces are treated, cf. Section \ref{sec_TL}. The proof uses Fourier series. With the help of admissible trivializations, we want to present a small generalization of this result for manifolds with $G$-actions. \\

We start by introducing our setup. In order to avoid any confusion with the metric $g$,  elements of the group $G$ are  denoted by $h$.

\begin{defi}[\bf $G$-manifold]
Let $(M,g)$ be a Riemannian manifold, and let $G$ be a countable discrete group that acts freely and properly discontinuously on $M$.  If, additionally, $g$ is invariant under the $G$-action (which means that $h: p\in M\mapsto h\cdot p\in  M$ is an isometry for all $h\in G$), we call $(M,g)$ a $G$-manifold. 
\end{defi}

By \cite[Corollary 12.27]{Lee} the orbit space $\widetilde{M}:=M/G$  of a $G$-manifold is again a  manifold. From now on we restrict ourselves to the case where $\widetilde{M}$ is closed. 
Let $\pi: M\to \widetilde{M}$ be the corresponding projection. If $(M,g)$ is a $G$-manifold, then there is a Riemannian metric $\tilde{g}$ on $\widetilde{M}$ such that $\pi^*\tilde{g}=g$. 
Let now $\widetilde{\uT}=(U_\alpha, \kappa_\alpha, h_\alpha)_{\alpha\in I}$ be an admissible trivialization of $\widetilde{M}$. In particular, this means  we assume that $(\widetilde{M},\tilde{g})$ is of bounded geometry  and, hence, so is $(M,g)$. Then there are $U_{\alpha,h}\subset M$ with $\pi^{-1}(U_\alpha)=\sqcup_{h\in G} U_{\alpha,h}$ and  $U_{\alpha,h}=h\cdot U_{\alpha,e}$ for all $\alpha\in I$. Here $e$ is the identity element of $G$. 
 Let $\pi_{\alpha,h}:=\pi|_{U_{\alpha,h}}: U_{\alpha,h}\to U_\alpha$ denote the corresponding diffeomorphism. 
 Setting $\kappa_{\alpha,h}:=\pi_{\alpha,h}^{-1}\circ \kappa_\alpha:V_\alpha \to U_{\alpha, h}$ and 
\begin{align*}                                                                                                                                                                                                         
h_{\alpha,h}:= \left\{ \begin{matrix} h_\alpha\circ \pi_{\alpha,h}  & \text{on\ }U_{\alpha,h},\\
                        0 & \text{else},\phantom{mm}
                       \end{matrix}
 \right.
\end{align*}

we have $h_{\alpha,h}\circ \kappa_{\alpha,h}= h_\alpha\circ \kappa_\alpha$ for all $\alpha\in I$, $h\in G$. This way we obtain an admissible trivialization  $\uT=(U_{\alpha,h}, \kappa_{\alpha,h}, h_{\alpha,h})_{\alpha\in I, h\in G}$  of $M$, which we call $G$-adapted trivialization. 

\begin{defi}[\bf $G$-adapted weight] Let $(M,g)$ be a $G$-manifold  with a $G$-adapted trivialization $\uT$ as above.
A weight function  $\rho: M \to (0,\infty)$ on $M$ is called $G$-adapted, if there exist a constant $C_{k}>0$  for all $k\in \mN_0$ such that for $\a\in \mN_0^n$ with $|\a|\leq k$ and all  $\alpha\in I$,  
\[
\sum_{h\in G} |\uD^\a(\rho \circ \kappa_{\alpha,h})|\leq C_k.
\]
\end{defi}

\begin{rem}\label{G_adap_triv}
The notion of a $G$-adapted weight is independent on the chosen admissible trivialization on $M/G$. This follows immediately from the compatibility of two admissible trivializations, cf. Remark \ref{rem_comp_T}.ii.
\end{rem}

\begin{ex}
We give an example of a weight adapted to the $G$-action.  Take a geodesic trivialization on $\widetilde{M}$ as in Example \ref{geod_coord_triv} and 
let $\uT$ be an admissible trivialization of $M$  constructed from  $\widetilde{\uT}$ on $\widetilde{M}$ as above.  
There is an injection $\iota: G\to \mN$, since $G$ is countable, and we set \[ \rho (p)=\sum_{(\alpha, h)\in I\times G;\ p\in U_{\alpha, h}} \iota(h)^{-2} h_{\alpha, h} (p).
\] 
 Since the covering is locally finite, the summation is always finite. Moreover,  Definition \ref{bddcoord} and the uniform finiteness of the cover yield for fixed $\alpha\in I$ and  all $\a\in \mN_0^n$ with $|\a|\leq k$ ($k\in \mN_0$), 

\begin{align*}
 \sum_{h\in G}|\uD^\a(\rho\circ \kappa_{\alpha,h})|\leq &\sum_{h\in G}\sum_{(\alpha', h')\in I\times G;\ \atop U_{\alpha,h}\cap U_{\alpha',h'} \neq \varnothing} \iota(h')^{-2} \left| \uD^\a  (h_{\alpha',h'}\circ \kappa_{\alpha, h})\right|\\
\leq & C_k'\sum_{h\in G}\sum_{|\a'|\leq |\a|}\ \sum_{(\alpha', h')\in I\times G;\ \atop  U_{\alpha,h}\cap U_{\alpha',h'} \neq \varnothing} \iota(h')^{-2} \left| \uD^{\a'} \left( h_{\alpha'} \circ \kappa_{\alpha'} \right)\right|\\
\leq & C_k''\sum_{h\in G}\sum_{(\alpha', h')\in I\times G;\ \atop U_{\alpha,h}\cap U_{\alpha',h'}\neq \varnothing} \iota (h')^{-2}
 =  C_k'' \sum_{h\in G} \sum_{(\alpha',h')\in I\times G,\ \atop U_{\alpha, e}\cap U_{\alpha',h^{-1}h'}\neq \varnothing} \iota(h')^{-2}\\
=  & C_k'' \sum_{h\in G} \sum_{(\alpha',h')\in I\times G,\ \atop  U_{\alpha, e}\cap U_{\alpha',h'}\neq \varnothing} \iota(hh')^{-2} 
\leq  C_k''L \sum_{h\in G} \iota(h)^{-2} \leq  C_k''L \sum_{i\in \mN} i^{-2}<\infty, 
\end{align*}
 where  $L$ is the multiplicity of the cover  and the constants $C_k', C_k''$ do not depend on $h\in G$ and $\alpha\in I$. 
In particular, together with Remark \ref{G_adap_triv}, this example demonstrates that each $G$-manifold admits a $G$-adapted weight.
\end{ex}

We fix some more notation. 
Let  $s\in \real$ and $1<p<\infty$. Then the space $H_p^s(M,\rho)$ consists of all distributions $f\in \mathcal{D}'(M)$ such that $$\Vert f\Vert_{H_p^s(M,\rho)}:=\Vert \rho f\Vert_{H_p^s(M)}<\infty.$$ Moreover, we call a distribution $f\in \mathcal{D}'(M)$ $G$-invariant, if $f(\phi)=f(h^*\phi)$ holds for all $\phi\in \mathcal{D}(M)$ and $h\in G$. The space of all $G$-invariant distributions in $H^s_p(M,\rho)$ is denoted by $H_p^s(M,\rho)^G$.

\begin{thm}\label{out_thm}
 Let $(M,g)$ be a $G$-manifold of bounded geometry where $\widetilde{M}=M/G$ is closed, and let $\rho$ be a $G$-adapted weight on $M$. Furthermore, let  $s\in \mathbb{R}$ and $1<p<\infty$. Then  $$H_p^s(\widetilde{M})=H_p^s(M,\rho)^G.$$ 
\end{thm}

\begin{proof}
 It suffices to show that the norms of the corresponding  spaces are equivalent. 
 We work with a geodesic trivialization $\widetilde{\uT}^{\rm geo}$ of $(M/G, \tilde{g})$ and a $G$-adapted trivialization $\uT$ of $M$ constructed from  $\tilde{T}^{\rm geo}$ as described above.  Note that the closedness of $M/G$ implies, that $\rho|_{\cup_\alpha U_{\alpha, e}}\geq c>0$ for some constant $c>0$ (since then $\cup_\alpha U_{\alpha, e}$ is compact). 
 Let $f'\in H_p^s(M/G)$ and set $f=f'\circ \pi$. Then,

\begin{align*}
 \Vert f\Vert_{H_p^s(M,\rho)}^p=&\sum_{\alpha\in I, h\in G} \Vert (h_{\alpha,h}\rho f)\circ \kappa_{\alpha,h}\Vert_{H_p^s(\mR^n)}^p= \sum_{\alpha\in I, h\in G} \Vert (\rho\circ \kappa_{\alpha,h})   \left( (h_\alpha f')\circ \kappa_{\alpha}\right) \Vert_{H_p^s(\mR^n)}^p\\
\lesssim &   \sum_{\alpha\in I} \Vert  (h_\alpha f')\circ \kappa_{\alpha} \Vert_{H_p^s(\mR^n)}^p =\Vert f'\Vert_{H_p^s(M/G)}^p.
\end{align*}

Let now $f\in H_p^s(M,\rho)^G$. Since $f$ is $G$-invariant,  there is a unique $f'$ with $f=f'\circ \pi$.
Then, 
\begin{align*}
 \Vert f'\Vert_{H_p^s(M/G)}^p=&\sum_{\alpha\in I} \Vert (h_{\alpha} f')\circ \kappa_{\alpha}\Vert_{H_p^s(\mR^n)}^p=\sum_{\alpha\in I} \Vert (\frac{1}{\rho}\circ \kappa_{\alpha,e})   \left( (h_{\alpha,e} \rho f)\circ \kappa_{\alpha, e}\right) \Vert_{H_p^s(\mR^n)}^p\\
 \lesssim & \sum_{\alpha\in I} \Vert \left( (h_{\alpha,e} \rho f)\circ \kappa_{\alpha, e}\right) \Vert_{H_p^s(\mR^n)}^p
 \leq \sum_{\alpha\in I, h\in G} \Vert \left( (h_{\alpha,h} \rho f)\circ \kappa_{\alpha, h}\right) \Vert_{H_p^s(\mR^n)}^p=\Vert f\Vert_{H_p^s(M,\rho)}^p.
\end{align*}
Here we used the uniform boundedness of $\frac{1}{\rho}\circ \kappa_{\alpha,e}$ and its derivatives, which follows from the corresponding statement for  $\rho\circ \kappa_{\alpha,e}$ and the lower bound $\rho\circ \kappa_{\alpha,e}\geq c>0$.
\end{proof}

\begin{rem} The restriction to closed  manifolds $\widetilde{M}$  (i.e., compact manifolds without boundary) 
in  Theorem \ref{out_thm} should not be necessary. In case that $\widetilde{M}$ is noncompact,  one needs to modify the definition of $G$-adapted weights in a suitable way to  assure  the weight is bounded away from zero with respect to the  {\em 'noncompact directions'} of $\widetilde{M}$. 
\end{rem}

We conclude our considerations with an example of a $G$-manifold other than the torus, which is covered by Theorem \ref{out_thm}. 

\begin{ex}
 Let $(\widetilde{M},\tilde{g})$ be a closed manifold. Let $G$ be  a subgroup of the fundamental group $\pi_1(\widetilde{M})$ of $\widetilde{M}$. Note that $G$ is countable since $\pi_1(\widetilde{M})$ is.  Let $(M,g)$ be the $G$-cover of $(\widetilde{M},\tilde{g})$ where $g=\pi^* \tilde{g}$. Then, $(M,g)$ is a $G$-manifold with $\widetilde{M}=M/G$.  
\end{ex}

\subsection{Triebel-Lizorkin spaces on manifolds}\label{sec_TL}

In order to keep our considerations as easy as possible, 
we have been concentrating on  (fractional) Sobolev spaces on Riemannian manifolds $M$ so far. This last paragraph is aimed at the reader who is more interested in the general theory of Besov and Triebel-Lizorkin spaces -- also referred to as B- and F-spaces in the sequel. 
We now want  to sketch how  those previous results generalize to Triebel-Lizorkin spaces  on manifolds. \\

By the Fourier-analytical approach, Triebel-Lizorkin spaces  $\F$, $s\in\real$, $0<p<\infty$,   $\ 0<q\leq\infty$, consist of 
all distributions $f\in \mathcal{S}'(\rn)$ such that
\begin{equation}\label{F-rn}
\big\Vert f\big\Vert_{\F}=
\Big\Vert \Big(\sum_{j=0}^{\infty}\big|2^{js}(\varphi_j\widehat{f})^\vee(\cdot)\big|^q\Big)^{1/q} 
\Big\Vert_{L_p(\rn)}
\end{equation}
$($usual modification if $q=\infty)$ is finite. 
Here  $\ \{\varphi_j\}_{j=0}^\infty$
denotes a   \textit{smooth dyadic resolution of unity}, where  $\ \varphi_0=\varphi \in \mathcal{S}(\rn)\ $ with 
$\ 
\supp\ \varphi\subset\left\{y\in\rn : \ |y|<2\right\}\quad \mbox{and}\quad
\varphi(x)=1\quad\mbox{if}\quad |x|\leq 1$,
and for each $\ j\in\nat\;$  put $\ \varphi_j(x)=
\varphi(2^{-j}x)-\varphi(2^{-j+1}x)$.  
The scale $\F$  generalizes fractional Sobolev spaces. In particular, we have the coincidence 
\[F^s_{p,2}(\rn)=H^s_{p}(\rn),\qquad s\in\mathbb{R}, \quad 1<p<\infty, \]
cf. \cite[p.~51]{T-F1}. In general, Besov spaces on $\rn$ are defined in the same way by interchanging the order in which the $\ell_q$- and $L_p$-norms are taken in \eqref{F-rn}. Hence,  the Besov space 
$B^s_{p,q}(\rn)$, $s\in\real$, $0<p,q \leq \infty$ consists of 
all distributions $f\in \mathcal{S}'(\rn)$ such that
\begin{align}
\big\Vert f\big\Vert_{B^s_{p,q}(\rn)}=
 \Big(\sum_{j=0}^{\infty}2^{jsq}\big\Vert(\varphi_j\widehat{f})^\vee\big\Vert_{L_p(\rn)}^q\Big)^{1/q} 
\end{align}
$($usual modification if $p=\infty$ and/or $q=\infty)$ is finite. In particular,  if $p=q$, 
$$B^s_{p,p}(\rn)=F^s_{p,p}(\rn), \qquad 0<p<\infty,$$
and we extend this to $p=\infty$ by putting $F^s_{\infty,\infty}(\rn):=B^s_{\infty,\infty}(\rn)$.  
The scales $F^s_{p,q}(\rn)$ and  $B^s_{p,q}(\rn)$ were studied in detail in  \cite{T-F1,Tri92}, where the reader may also find further references to the literature. \\

On $\mR^n$ one usually gives priority to Besov spaces, and they are mostly considered to be the simpler ones compared to Triebel-Lizorkin spaces. However,  the situation is different on manifolds $M$, since B-spaces lack the so-called \textit{ localization principle}, cf. \cite[Theorem~2.4.7(i)]{Tri92}, which is used to define  
F-spaces on $M$ (as was already done in  Definition \ref{H-koord} for  fractional Sobolev spaces, now replacing $H^s_p(\rn)$ by $F^s_{p,q}(\rn)$ inside of the norm). Then Besov spaces on $M$ are introduced via real interpolation of Triebel-Lizorkin spaces (in order to compute traces we have to generalize  the B-spaces on $M$ from Definition {\ref{def-B-spaces}} and allow $0<p\leq 1$).

\begin{defi}\label{F-koord}
Let $(M^n,g)$ be a Riemannian manifold  with an admissible trivialization $\uT=(U_\alpha, \kappa_{\alpha}, h_{\alpha})_{\alpha \in I}$ and  let $s\in \real$.

\begin{itemize}
\item[(i)] Let either  $0<p<\infty$, $0<q\leq \infty$ or $p=q=\infty$.  Then the space $F^{s,\uT}_{p,q}(M)$ contains all distributions $f\in \mathcal{D}'(M)$ such that 
\beq\label{F-def}
\Vert f\Vert_{F^{s,\uT}_{p,q}(M)}:=\left(\sum_{\alpha\in I} \Vert (h_\alpha f)\circ \kappa_\alpha\Vert^p_{F^s_{p,q}(\mathbb{R}^n)}\right)^{\frac{1}{p}}
\eeq
is finite (with the usual modification if $p=\infty$). 
\item[(ii)] Let $0<p,q\leq \infty$,  and let $-\infty< s_0<s< s_1<\infty$. Then 
\[
B^{s,\uT}_{p,q}(M)=\left(F^{s_0,\uT}_{p,p}(M), F^{s_1,\uT}_{p,p}(M)\right)_{\Theta,q}
\]
with $s=(1-\Theta)s_0+\Theta s_1$. 
\end{itemize}
\end{defi}

\begin{rem} Restricting ourselves to  geodesic trivializations $\uT^{\text{geo}}$, the spaces from Definition \ref{F-koord} coincide with the spaces $F^s_{p,q}(M)$ and $B^s_{p,q}(M)$, introduced in \cite[Definition~7.2.2, 7.3.1]{Tri92}. 
The space $B^{s,\uT}_{p,q}(M)$ is  independent of the chosen numbers $s_0, s_1\in \real$ and, furthermore, for $s\in \real$ and $0< p\leq \infty$ we have the coincidence 
\beq\label{B-norm}
B^{s,\uT}_{p,p}(M)=F^{s,\uT}_{p,p}(M).
\end{equation}
This follows from \cite[Theorem~7.3.1]{Tri92}, since the arguments presented there  are based on interpolation and  completely oblivious of the chosen trivialization $\uT$. In particular,  \eqref{B-norm} yields that 
for 
 $f\in B^s_{p,p}(M)$ a quasi-norm  is given by
\beq\label{B-def-2}
\Vert f\Vert_{B^{s,\uT}_{p,p}(M)}=\left(\sum_{\alpha\in I} \Vert (h_\alpha f)\circ \kappa_\alpha\Vert^p_{B^s_{p,p}(\mathbb{R}^n)}\right)^{\frac{1}{p}}.
\eeq
\end{rem}

Now we can transfer Theorem \ref{indep_H} to F- and B-spaces. 

\begin{thm}\label{indep-F}
Let $(M^n,g)$ be a Riemannian manifold with an admissible trivialization $\uT=(U_{\alpha},\kappa_{\alpha}, h_{\alpha})_{\alpha\in I}$. Furthermore, let $s\in \real$ and let $0<p,q\leq \infty$ 
$(0<p,q<\infty$ or $p=q=\infty$ for F-spaces$)$. Then 
\[
F^{s,\uT}_{p,q}(M)=F^{s}_{p,q}(M) \qquad 
\text{and}\qquad  
B^{s,\uT}_{p,q}(M)=B^{s}_{p,q}(M).
\]
\end{thm}

\begin{proof}
 For  F-spaces the proof is the same as the one of Theorem \ref{indep_H}. The claim for  B-spaces then follows from Definition \ref{F-koord}.ii.
\end{proof}

\paragraph{\bf Trace theorem.} The generalization of the Trace Theorem \ref{trace-th} is stated below. In particular, this result improves \cite[Theorem~1, Corollary~1]{Skr90}. 

\begin{thm}\label{trace-th-gen} Let $(M^n,g)$ be a Riemannian manifold together with an embedded $k$-dimensional submanifold $N$ and 
 $(M^n,N^k)$ be of bounded geometry. Furthermore, let $0<p,q\leq \infty$ $(0<p,q< \infty$ or $p=q=\infty$ for F-spaces$)$ and let  
\begin{equation}\label{cond-s}
s-\frac{n-k}{p}>k \left(\frac 1p-1\right)_+.
\end{equation} 
{Then $\Tr_N=\Tr$ is a linear and bounded operator from  $F^s_{p,q}(M)$ onto  $B^{s-\frac{n-k}{p}}_{p,p}(N)$ and $B^s_{p,q}(M)$ onto  $B^{s-\frac{n-k}{p}}_{p,q}(N)$, respectively, i.e., 
\beq\label{trace-mfd-gen}
\Tr_N\, F^s_{p,q}(M)=B^{s-\frac{n-k}{p}}_{p,p}(N)
\qquad \text{and}\qquad 
\Tr_N\, B^s_{p,q}(M)=B^{s-\frac{n-k}{p}}_{p,q}(N).
\eeq}
\end{thm}

\begin{proof} The proof of  \eqref{trace-mfd-gen} runs along the same lines as the proof of Theorem \ref{trace-th}. Choosing Fermi coordinates, via pull back and localization the problem can be reduced to corresponding trace results in $\rn$ on hyperplanes $\mR^{k}$, cf. \cite[Theorem~4.4.2]{Tri92}, where the proof for $k=n-1$ may be found.  The result for general hyperplanes  -- and condition \eqref{cond-s} --  follows by iteration.  
The assertion for B-spaces follows then from  Definition \ref{F-koord}.ii.
\end{proof}

\bibliographystyle{alpha}

\begin{thebibliography}{Jia93}

\bibitem[Aub76]{Aub1}
T.~Aubin.
\newblock {Espaces de Sobolev sur les vari\'et\'es Riemanniennes}.
\newblock {\em Bull. Sci. Math.} 100: 149--173, 1976. 

\bibitem[Aub82]{Aub2}
T.~Aubin.
\newblock {\em Nonlinear Analysis on Manifolds}.
\newblock {Monge-Amp\`ere Equations}, New York, Springer, 1982. 


\bibitem[Eic07]{Eichb}
J.~Eichhorn. 
\newblock {\em Global analysis on open manifolds}.
\newblock {Nova Science Publishers, Inc., New York} x+644,  2007.

\bibitem[Eic91]{Eich}
J.~Eichhorn.
\newblock {The boundedness of connection coefficients and their derivatives}.
\newblock {\em Math. Nachr.} 152: 144--158, 1991. 

\bibitem[Fr00]{Frie}
T.~Friedrich.
\newblock {\em Dirac operators in {R}iemannian geometry}.
\newblock {Graduate Studies in Mathematics, American Mathematical Society, Providence, RI} 25: xvi+195, 2000.

\bibitem[Gro12]{Gro}
N.~Gro\ss e.
\newblock {Solutions of the equation of a spinorial {Y}amabe-type problem on manifolds of bounded geometry}.
\newblock {\em Comm. Part. Diff. Eq.} 37(1): 58--76 2012. 

\bibitem[GN12]{GN}
N.~Gro\ss e, R.~Nakad.
\newblock {Boundary value problems for noncompact boundaries of Spin$^c$ manifolds and spectral estimates}.
\newblock arXiv:1207.4568

\bibitem[Lee01]{Lee}
J.M.~Lee.
\newblock {\em Introduction to topological manifolds}.
 Second edition,    
 \newblock {Graduate Texts in Mathematics} (202), Springer, New York, 2011.
 
\bibitem[Sch01]{Schick01}
T.~Schick.
\newblock {Manifolds with boundary and of bounded geometry}.
\newblock {\em Math. Nachr.}, 223:103--120, 2001.

\bibitem[Shu]{Shu} M.A.~Shubin,
\newblock {Spectral theory of elliptic operators on noncompact manifolds}. 
\newblock {\em M{\'e}thodes semi-classiques, Vol. 1 $($Nantes, 1991$)$, Ast\'erisque} 207:35--108
1992.

\bibitem[Skr90]{Skr90}
L.~Skrzypczak.
\newblock {Traces of Function Spaces of $F^s_{p,q}$--$B^s_{p,q}$ Type on Submanifolds}.
\newblock {\em Math. Nachr.} 46: 137--147, 1990. 

\bibitem[Skr98]{Skr98}
L.~Skrzypczak.
\newblock {Atomic decompositions on manifolds with bounded geometry}.
\newblock {\em Forum Math.} 10: 19--38, 1998. 

\bibitem[SpIV]{Spiv4}
M.~Spivak.
\newblock {\em A comprehensive introduction to differential geometry, {V}ol.{IV}}.
\newblock {Second edition,  Publish or Perish Inc. Wilmington, Del.} viii+561, 1979.

\bibitem[Str83]{strich}
R.S. Strichartz.
\newblock {Analysis of the Laplacian on the complete Riemannian manifold}.
\newblock {\em J. Funct. Anal.} 52: 48--79, 1983. 

\bibitem[Tr86]{TrB}
H.~Triebel.
\newblock {Spaces of {B}esov-{H}ardy-{S}obolev type on complete {R}iemannian manifolds}.
\newblock{\em Ark. Mat.} 24(2): 299--337, 1986.

\bibitem[Tri83]{T-F1}
H.~Triebel.
\newblock {\em Theory of function spaces}, volume~78 of {\em Monographs in
  Mathematics}.
\newblock Birkh\"auser Verlag, Basel, 1983.

\bibitem[Tri92]{Tri92}
H.~Triebel.
\newblock {\em Theory of function spaces {II}}, volume~84 of {\em Monographs in
  Mathematics}.
\newblock Birkh\"auser Verlag, Basel, 1992.

\bibitem[Tri78]{T-interpol}
H.~Triebel.
\newblock {\em Interpolation theory, function spaces, differential operators}, volume 18.  
\newblock { North-Holland Publishing Co.}, Amsterdam, 1978.

\bibitem[Wa66]{War}
F.W.~Warner.
\newblock {Extensions of the {R}auch comparison theorem to submanifolds}.
\newblock {\em Trans. Amer. Math. Soc.} 122: 341--356, 1966.
\end{thebibliography}

 \def\cprime{$'$}

\vfill

{\small
\begin{minipage}[t]{0.45\textwidth}
\noindent
Nadine Gro\ss{}e\\
Mathematical Institute\\  
 University of  Leipzig \\
 Augustusplatz 10\\ 
 04109 Leipzig \\
 Germany\\[1ex]
{\tt grosse@math.uni-leipzig.de}
\end{minipage}\hfill 
\begin{minipage}[t]{0.45\textwidth}
\noindent
Cornelia Schneider\\
Applied Mathematics III\\
University of Erlangen--Nuremberg\\
Cauerstra\ss{}e 11\\
91058 Erlangen\\
Germany\\[1ex]
{\tt schneider@math.fau.de}
\end{minipage}

}

\end{document}